\documentclass[11pt,
]{amsart}
\usepackage{
amssymb, graphicx
}
\usepackage{mathrsfs} 
\usepackage{amsthm}
\usepackage{graphicx,color}
\usepackage{hyperref}
\usepackage{subcaption}

\newtheorem{theorem}{Theorem}[section]
\newtheorem{lemma}[theorem]{Lemma}

\newtheorem{definition}[theorem]{Definition}
\newtheorem{remark}[theorem]{Remark}
\newtheorem{problem}[theorem]{Problem}

\title[leakage, collision, and self-adjoint cancellation]{Leakage detection, collision relation, and self-adjoint cancellation in multi-velocity systems}
\author{Yuchao Yi}
\date{}

\begin{document}
\begin{abstract}
Let $W^T$ map Dirichlet data $f$ to time $T$ state $u^f(T, \cdot)$, we study the $K$-normal operator $C^T_K:= (W^T)^*KW^T$ where $(W^T)^*$ is the standard $L^2(dx)$ adjoint. We prove that it is a locally finite sum of Fourier Integral Operators (FIOs) away from the glancing directions with canonical relation involving collision data at time $T$ between pairs of velocities. However, if the operator is $M$-self-adjoint, then the canonical relation for $C^T_M$ loses the collision data. Thus, when system satisfies certain coupling requirement, we demonstrate off-polarization leakage detection and introduce a new collision rigidity problem. Combine these two parts, we prove an inverse problem of velocity recovery from $C^T_I$ for multi-velocity wave models and variable coefficient isotropic elasticity system, where $I$ is the identity matrix.
\end{abstract}

\maketitle

\tableofcontents

\section{Introduction}

Many hyperbolic systems have several characteristic families. For instance, in the isotropic Lam\'e system, wave propagation splits into pressure and shear waves. More generally, the principal symbol of a system may have several real eigenvalues, each giving a different characteristic sheet and a different polarization space. The purpose of this paper is to study what information about the interaction between these characteristic sheets can be read and used for inverse problems.

Let $E=\Omega\times \mathbb C^N$ be a trivial vector bundle over a compact domain $\Omega$ with smooth boundary. We consider second order hyperbolic systems of the form
\[
\mathcal L u := \mathsf m(x)\partial_t^2 u + P(x,D_x)u = 0,
\]
where $\mathsf m(x)$ is pointwise invertible and $P$ is a second order pseudodifferential operator. After multiplying by $\mathsf m^{-1}$, one may work microlocally with a normalized system
\[
Lu := \partial_t^2 u + A(x,D_x)u=0.
\]
Let $a_2(x,\xi)$ be the principal symbol of $A$. We assume that $a_2$ is smoothly diagonalizable and the eigenvalues are positive, real, and have constant multiplicities. Denote them by
\[\lambda_1^2 < \cdots < \lambda_k^2.\]
Then the characteristic variety decomposes into disjoint sheets
\[
\Sigma_\alpha=\{(t,x,\tau,\xi):\tau^2=\lambda_\alpha^2(x,\xi)\},
\qquad \alpha=1,\ldots,k.
\]
Each sheet carries a polarization space, and singularities associated with different sheets propagate with different speeds. We assume that after first order reduction, the \emph{uniform Lopatinski condition} (see Definition 1.2 of \cite{Met17}) is satisfied on the boundary so that well-posedness holds with any $L^2$ Dirichlet boundary data, see Theorem 1.8 of \cite{Met17} (the theorem holds under constant multiplicity assumption, see discussion before Theorem 1.8).

For $T>0$, let $W^T$ be the solution operator mapping Dirichlet boundary data $f$ on $(0,T)\times\partial\Omega$ to the time $T$ state $u^f(T,\cdot)$. Given a smooth positive bundle endomorphism $K$, we define the $K$-normal operator by
\[
C_K^T=(W^T)^*KW^T,
\]
where $(W^T)^*$ is the adjoint with respect to the standard $L^2(dx)$ pairing in the interior and the standard $L^2(dt,dS)$ pairing on the boundary. Equivalently,
\[
\langle C_K^T f,g\rangle
=
\int_\Omega (u^f(T,x))^*K(x)u^g(T,x),dx .
\]
One of the main goals of the paper is to determine the microlocal structure of $C_K^T$.

For simplicity, let us first restrict our attention to the set of rays that reach time $T$ without going through any reflection. One may expect the same-sheet part of $C_K^T$, namely the interaction between two $\alpha$-beams, to encode the scattering relation. A different feature in the multi-velocity setting comparing to one-velocity system is that, for a general weight $K$, the operator $C_K^T$ may also contain cross-sheet information: that is, the interaction between an $\alpha$-beam and a $\beta$-beam at time $T$. Microlocally, these terms have canonical relation of the form $C_\beta^t\circ C_\alpha$, where $C_\alpha$ and $C_\beta$ are the canonical relations mapping boundary covectors to the time $T$ slice along the $\alpha$- and $\beta$-characteristic flows. Therefore $C_K^T$ may contain information about pairs of rays belonging to different characteristic sheets which arrive at the same point and covector on the time slice. However, such cross-sheet interaction may not be visible: in particular, we prove that it is hidden for $C^T_M$ where the operator $A$ is self-adjoint with respect to $M$. We prove this in Section \ref{sec: normal operator for general multi velocity system}, see Theorem \ref{thm: microlocal} for the microlocal characterization of $K$-normal operators and Remark \ref{rmk: self adjoint too strong} for some further discussion.

To also include beams that have reflected, we need to discuss about mode conversion for multi-velocity systems. Mode conversion refers to the phenomenon that, when an incoming $\alpha$-beam reaches the boundary, it may generate not only the $\alpha$-beam reflection according to Snell's Law, but also an outgoing $\beta$-beam. This mode conversion may happen for each reflection for any beam, which increases the number of beams exponentially. We refer to Figure \ref{fig: mode conversion} for an illustration of mode conversion, Section \ref{sec: mode conversion} for related discussion, and \cite{Tay75, HU03, SUV18} for its affect on Dirichlet-to-Neumann map. We emphasize that mode conversion and cross-sheet interaction are two different behaviors: mode conversion only occurs at the lateral boundary when the metric is smooth, while the cross-sheet interaction at time $T$ slice refers to whether two beams corresponding to different characteristic sheets interact or not. Mode conversion does make the microlocal property of $C^T_K$ more complicated, but we focus on studying how different choices of $K$ affect the cross-sheet interaction, which determines the microlocal property of $C^T_K$ in a more fundamental way.

\subsection{Motivation}
The study of $C_K^T$ is motivated by the Boundary Control method. For the scalar wave equation $(\partial_t^2 - \Delta_g)u=0$, the operator
\[
C^T=(W^T)^{*_g}W^T
\]
where the adjoint $*_g$ is with respect to the standard volume form $dV_g$, is usually called the \emph{connection operator}. Through the Blagoveshchenskii identity (see for example \cite{Bla69, Bel07}), it can be determined from the Dirichlet-to-Neumann map, hence revealing the inner products of waves at time $T$ without observing the solution in the interior. In the classical Boundary Control method, this operator is used together with approximate controllability to construct the wave model of the manifold, we refer to \cite{Bel07} for more details on Boundary Control method.

For multi-velocity systems the situation is less direct. Different characteristic families propagate with different speeds, so the usual controllability result no longer gives an immediate description of all components of the time $T$ state. In particular, there may be regions which are reached by a faster wave before they are reached by a slower one, and standard controllability arguments do not by themselves explain what state can be reached there. One of the most important such systems is the isotropic elasticity system (also known as Lam\'e system):
\[
L_{\lambda,\mu,\rho}u
:=
\rho \partial_t^2 u-\nabla\cdot S^L(x,u)=0,
\]
where $\rho(x), \lambda(x), \mu(x)$ are variable coefficients and
\[
S^L_{ij}(x,u)
=
\lambda \widetilde\varepsilon_{mm}(u)\delta_{ij}
+
2\mu \widetilde\varepsilon_{ij}(u),
\qquad
\widetilde\varepsilon_{ij}(u)
=
\frac12(\partial_i u_j+\partial_j u_i).
\]
The system naturally has a two-velocity structure
\[
c_P^2 = \frac{\lambda+2\mu}{\rho}, \quad c_S^2 = \frac{\mu}{\rho}.
\]
In \cite{BL02}, the authors proved approximate controllability in the slow wave region, but the proof does not directly work for the fast wave region. We shall discuss in more details about this system in Section \ref{sec: application to lame system} as it is one of the main motivations. We refer to \cite{Bel07, BL02} for a more thorough discussion on the lack of approximate controllability for multi-velocity systems. For other related papers on Lam\'e system not mentioned below, see \cite{HUV20, HNZ17, HNZ19, Zou24, LMOSZ26}.

This difficulty motivates a more microlocal question: rather than first proving controllability and then constructing the wave model, one may ask what geometric information is already contained in the operator $C_K^T$ itself. The main observation of this paper is that, for multi-velocity systems, $C_K^T$ with suitable $K$ contains data related to how two velocities interact with each other, but this information is masked when the operator is self-adjoint with the pairing.

Microlocal analysis has been extensively applied to various inverse problems, including multi-velocity systems such as the Lam\'e system. A common strategy is to use the fact that microlocally one can track the propagation along each characteristic sheet, and then reduce the inverse problem to standard one-velocity problems for each sheet. In this approach, the main geometric data extracted from the boundary measurements are the lens relation, scattering relation, or boundary distance functions associated with the individual velocities. The subsequent inverse problem is then a lens rigidity or boundary rigidity problem for each characteristic family. Thus the coupling structure of the system is treated mainly as something to be removed, so that one may apply the corresponding scalar theory. For microlocal-based approach on Lam\'e system, we refer to \cite{Rac00, HU03, SUV18, ISY26}. For lens and boundary rigidity results, we refer to \cite{SU08, SUV21b}. For inverse problem related to nonlinear elasticity equation which is also based on microlocal method, we refer to \cite{UZ21, UZ24}.

The point of view in this paper is different. We still use microlocal diagonalization, but not only to decouple the system. Rather, we use it to identify precisely which part of the normal operator comes from same-sheet propagation and which part could come from interactions between different sheets. The latter is invisible at the principal polarization level, but it may appear through lower order off-polarization leakage. In this sense, the coupling is not an obstacle, it is a possible source of additional geometric information.

This distinction is also consistent with the role of self-adjointness in the usual Dirichlet-to-Neumann inverse problem. For a self-adjoint hyperbolic system, the Blagoveshchenskii identity determines the natural connection operator, namely the operator $C_M^T$ corresponding to the self-adjoint density $M$. Our microlocal analysis shows that for this particular choice of pairing the cross-sheet canonical relations cancel: $C_M^T$ contains only the same-sheet relations $C_\alpha^t\circ C_\alpha$. Therefore one should not expect the standard Dirichlet-to-Neumann map of a self-adjoint system to reveal cross-sheet collision data.

\subsection{Leakage detection and directional collision}
We now explain how the collision data related to the interaction of two characteristic sheets can be detected and used. A Gaussian beam solution is an approximate solution highly concentrated near some chosen geodesic, and we denote it $u^*_h$, where $*$ will indicate which metric the geodesic belongs to and $h$ is the parameter controlling how concentrated the solution is. For a multi-velocity system, one can thus construct Gaussian beam solutions with respect to different metrics. Suppose there are $g_P$ and $g_S$ two Riemannian metrics, then denote $u^P_h$ and $u^S_h$ the Gaussian beam solutions centered around $\gamma_P$ and $\gamma_S$ geodesics. Since $C^T_K$ is essentially an interior scan, we want to use it to detect whether $\gamma_P$ and $\gamma_S$ released from the boundary at time $t_P$ and $t_S$, respectively, collide or not at time $T$. Specifically, since the solutions are highly concentrated near the ray, one natural expectation is that the $h$-decay rate is fast when the two beams do not collide, and slow when the two beams do collide. In fact, as the Gaussian beam solution is essentially a highly oscillatory solution constructed with respect to a chosen bicharacteristic curve, the inner product measures not just the space distance but also the phase distance. Hence, we should expect non-trivial decay rate only when the two geodesics $\gamma_P$ and $\gamma_S$ intersect at the same location and point in the same direction. Note that this is useful only in multi-velocity system: if there is only one metric, a given point and direction uniquely determines a geodesic, so a geodesic can only collide with itself. There are two obstructions to this naive testing.

The first one being that, for a two-velocity system with distinct wave speeds, the characteristic sets are disjoint at each point. In other words, two geodesics related to different metrics can never share the same spacetime location and spacetime frequency. What makes stationary phase possible is a spatial $L^2$ scan instead of a spacetime $L^2$ interaction. In other words, even though the spacetime frequency differs, it is possible for their spatial projection to align, which gives stationary phase in the spatial $L^2$ inner product.

The second obstruction, is that this stationary phase may not provide any useful information because of the vanishing of amplitudes. For a two-velocity model with characteristic sheets corresponding to $P$- and $S$-metric, the $P$-polarization set and $S$-polarization set have trivial inner product (polarization set is, roughly speaking, the most singular direction, we refer to \cite{Den82} for definition). However, note that this is only a strong restriction for the principal amplitude: starting from subprincipal amplitude, there may not be such requirement. In particular, there is off-polarization leakage for subprincipal symbol, provided certain coupling exists. In other words, denoting the amplitudes as $a^*_0 + ha^*_1 + \cdots$, we expect leakage to be detected not from $a_0^P \cdot a_0^S$, but starting from principal-subprincipal interaction
\[
L_{01} = a_0^P \cdot a_1^S + a_1^P \cdot a_0^S,
\]
together with the second order term in the stationary phase expansion of the principal-principal interaction $L_{00}$.

Combine together, to have (first order) detectable off-polarization leakage, one need to compute
\[
L_{00} + L_{01}
\]
and show that it is non-vanishing. Certainly, one can consider second order interactions $L_{11}+L_{02} + L_{01} + L_{00}$ and even higher order ones.

To connect with the microlocal result we obtained for $C^T_K$, note that the interaction is $K$-dependent. In particular, if the operator is self-adjoint with respect to $M$, then for all orders the total leakage will cancel out perfectly. This is essentially because the system can be microlocally decoupled, and $M$ is also microlocally ``decoupled'' by the same procedure, so the system with respect to $M$-pairing is fully (microlocally) decoupled.

Back to the leakage detection, if $K$ is such that the operator is not self-adjoint with respect to $K$, then one can compute the required leakage detection requirement. In this paper, we explicitly compute the first order leakage detection result for multi-velocity wave systems in Section \ref{sec: toy model} and Lam\'e system in Section \ref{sec: application to lame system} when $K = I$ is the identity matrix, namely the standard $L^2(dx)$ pairing.

The result of leakage detection is a much richer set of geometric data: the directional collision relation (see Definition \ref{def: directional collision relation}). Namely, one has the collision data with respect to $P$-geodesics and $S$-geodesics, instead of the geometric data for each individual metric. This is, arguably, a more natural data for multi-velocity systems because it respects the fact that the system contains several velocities. To the author's knowledge, this collision rigidity problem has not been studied before. When one has several metrics, we study whether this collision data for each pair of metric uniquely determine these metrics in Section \ref{sec: directional collision rigidity problem}.

\subsection{Overview and main results}
In Section \ref{sec: normal operator for general multi velocity system}, we study the microlocal property of $C^T_K$ for different $K$ including the cancellation result when the operator is self-adjoint, see Theorem \ref{thm: microlocal}. Then in Section \ref{sec: toy model}, we demonstrate the off-polarization leakage detection method for wave systems because their polarization sets are fixed for all point and direction, see Theorem \ref{thm: leakage detection} and Theorem \ref{thm: k leakage detection} for the leakage detection results. As the leakage detection result provides directional collision relation, we study related rigidity problem in Section \ref{sec: directional collision rigidity problem}, and the main results are Theorem \ref{thm: 1 metric known} and Theorem \ref{thm: k collision rigidity}. In Section \ref{sec: wave system inverse problem}, we combine leakage detection and collision rigidity to prove an inverse problem for the multi-velocity wave system, see Theorem \ref{thm: k wave inverse problem}. Finally, we apply leakage detection to the Lam\'e system as it is one of the main motivation of this paper and prove an inverse problem, see Theorem \ref{thm: lame}. We also refer to Section \ref{sec: cancellation wave} and Section \ref{sec: cancellation lame} for discussion on self-adjoint cancellation for multi-velocity wave system and Lam\'e system, respectively.

\section*{Acknowledgment}
The author would like to thank Wenkui Du, Katya Krupchyk, Teemu Saksala, Gunther Uhlmann and Lili Yan for helpful discussions. The author would also like to thank Teemu Saksala for introducing Lam\'e system. The Gaussian beam computations in Section \ref{sec: application to lame system} is assisted by ChatGPT-5.5. The discussions with ChatGPT-5.5 about cancellation of $\rho$-weighted inner product is very helpful, which inspired the more general microlocal computation in Section \ref{sec: normal operator for general multi velocity system}.

\section{Normal operator for general multi-velocity system}\label{sec: normal operator for general multi velocity system}
Let $E=\Omega\times \mathbb C^N$ be the trivial vector bundle over $\Omega\subset \mathbb R^n$. Consider a second order hyperbolic system
\[
    \mathcal L u = \mathsf m(x)\partial_t^2 u + P(x,D_x)u ,
\]
where $\mathsf m(x) \in C^\infty(\Omega; \mathrm{End}(E))$ is pointwise invertible and $P$ is a smooth second order pseudodifferential operator. We may multiply the equation by $\mathsf m^{-1}$ and work instead with the normalized system
\[
    L u = \partial_t^2 u + A(x,D_x)u ,
    \qquad A=\mathsf m^{-1}P .
\]
We do not assume at this stage that $A$ is self-adjoint.

Let $a_2(x,\xi)=\sigma_2(A)(x,\xi)$ be the principal symbol of $A$, we assume $a_2(x, \xi)$ is smoothly diagonalizable and the eigenvalues are real and positive, denote by $\lambda_1^2(x,\xi) < \cdots < \lambda_k^2(x, \xi)$ with multiplicity $m_1, \cdots, m_k$. We require that the multiplicity of each eigenvalue is constant on the entire $T^*\Omega$. The principal symbol of $L$ is
\[
    \ell_2(t,x,\tau,\xi)
    =
    -\tau^2 I + a_2(x,\xi).
\]
By the constant multiplicity requirement, the characteristic set splits microlocally into smooth disjoint characteristic sheets
\[
    \Sigma=\{(t,x,\tau,\xi): \det(-\tau^2 I+a_2(x,\xi))=0\} =\bigcup_{\alpha=1}^k \Sigma_\alpha,
\]
where
\[
    \Sigma_\alpha
    =
    \{(t,x,\tau,\xi): \tau^2=\lambda^2_\alpha(x,\xi)\}.
\]
The sheets do not touch each other because
\[
    \lambda^2_\alpha(x,\xi)\neq \lambda^2_\beta(x,\xi),
    \qquad \alpha\neq \beta,\quad \xi\neq0.
\]
Each branch determines a Hamiltonian
\[
    p_\alpha(t,x,\tau,\xi)
    =
    \tau^2-\lambda^2_\alpha(x,\xi).
\]
Note that they do not necessarily induce a Lorentzian metric structure, unless $\lambda^2_\alpha$ is quadratic polynomial in $\xi$. We assume the uniform Lopatinski condition holds for the reduced first order system, so that the system is well-posed (see \cite[Theorem 1.8]{Met17}). For some of the results below, we will include the following extra assumptions:
\begin{itemize}
    \item \textbf{(H)}: $\lambda_\alpha^2$ is quadratic polynomial in $\xi$ with $g_\alpha$ the corresponding Riemannian metric; the boundary $\partial \Omega$ is strictly convex with respect to all $g_\alpha$; all the geodesics for $g_\alpha$ are non-trapping.
\end{itemize}
The above assumption simplifies the argument in the following ways:
\begin{enumerate}
    \item having Riemannian metric structure is common in multi-velocity systems such as the wave system and Lam\'e system we deal with in this paper, and it allows us to use terms related to geodesics;
    \item strict convexity allows us not to deal with complicated glancing behaviors;
    \item non-trapping is usually required for metric recovery results, and we do not need to mention whether a geodesic reaches a boundary or not every time we backtrack along a geodesic.
\end{enumerate}
Certainly, each of the above assumptions can be weakened and dealt with through more explicit computations, and adding these assumptions does not affect the general theme of the paper.

For each branch $\Sigma_\alpha$, define the polarization space
\[
    E_\alpha(t,x,\tau,\xi)
    =
    \ker \ell_2(t, x, \tau, \xi),
    \qquad (t,x,\tau,\xi)\in \Sigma_\alpha .
\]

Again denote $W^T(f):= u^f(T, \cdot)$ the solution with boundary data $f$ at time $T$. Let $K \in C^\infty(\Omega; \mathrm{End}(E))$ be positive definite, and define $K$-weighted inner product as
\[
\langle u, v\rangle_K := \int_\Omega u^*Kv dx
\]
where $u^* = \overline{u}^t$ is the conjugate transpose. Then we denote $(W^T)^*$ as the adjoint operator with respect to standard $L^2(dx)$ inner product, and the $K$-normal operator as
\[
C^T_K := (W^T)^*KW^T, \quad \langle C^T_K f, g\rangle_{[0,T]\times \partial \Omega} = \int_\Omega u^f(T, x)^*K(x)u^g(T, x)dx.
\]

\subsection{Microlocal decoupling}

To study how $K$ determines the microlocal property of $C^T_K$, we start with a microlocal decoupling of the system, the construction follows essentially \cite{Tay75, Tay81}. Terminology-wise, we say a matrix $B$ is \emph{block diagonal with respect to the characteristic sheets}, when it is of the form
\[
B = \begin{pmatrix}
    B_1 & 0 & \dots & 0 \\
    0 & B_2 & \dots & 0 \\
    \vdots & \vdots & \ddots & 0\\
    0 & \dots & 0 & B_k
\end{pmatrix}
\]
where $B_j$ is an $m_j \times m_j$ matrix, and $m_j$ is the multiplicity for the eigenvalue $\lambda_j^2$. We use $\Psi^k$ to denote the set of pseudodifferential operators of order $k$; for a matrix pseudodifferential operator, we follow the notations in \cite{Ste00} and use $\Psi^{\alpha, \beta}$ to indicate that it is of order $\alpha$ for the block entries, and order $\beta$ for the off-block entries.

\begin{lemma}\label{lem: decomposition}
    Let $A(x, D_x) \in \Psi^2(\Omega; \mathrm{End}(E))$ be such that the principal symbol is smoothly diagonalizable with real positive eigenvalues, and the characteristic sheets for the multi-velocity system are disjoint with constant multiplicity. There exists elliptic $U \in \Psi^0(\Omega; \mathrm{End}(E))$ with parametrix $U^{-1}$ such that $B:= U^{-1}AU$ is block diagonal with respect to the characteristic sheets modulo $\Psi^{-\infty}$. Moreover, let $M \in C^\infty(\Omega; \mathrm{End}(E))$ be positive-definite, then $A$ is $M$-self-adjoint, namely $A^*M - MA = \Psi^{-\infty}$ where $A^*$ is the standard adjoint with respect to $L^2(dx)$, if and only if
    \[
    B^*U^*MU -  U^*MUB= \Psi^{-\infty}.
    \]
\end{lemma}
\begin{proof}
    We first assume such $U$ exists. For the reverse direction, suppose $U$ and $B$ are as stated. Then
    \[
    A^*M = (UBU^{-1})^*MUU^{-1} = (U^{-1})^*(U^*MU)BU^{-1} = MA \quad \mathrm{mod} \quad \Psi^{-\infty}.
    \]
    For the forward direction, suppose $A$ is $M$-self-adjoint, then
    \[
    B^*U^*MU = (UB)^*MU = U^*A^*MU = U^*MAU = U^*MUB \quad \mathrm{mod} \quad \Psi^{-\infty}.
    \]
    
    Now we construct $U$ microlocally. Denote $a_2(x, \xi)$ the principal symbol of $A$, the eigenvalues $\lambda_1^2, \cdots, \lambda_k^2$ are all real and positive with constant multiplicity $m_\alpha$ on $T^*\Omega$. By our assumption, $a_2$ is diagonalizable.
    For every $(x, \xi)$, by eigendecomposition we can find $u_0(x, \xi)$ such that $b_2 := u_0^{-1}a_2u_0 = \mathrm{diag}(\lambda_1, \cdots, \lambda_N)$.
    Even though $b_2$ is diagonal, the final operator we obtain will only be block diagonal, where the blocks are given by the dimension of eigenspaces.
    Recall we denote $\Psi^{\alpha, \beta}$ the set of pseudodifferential operators whose block entries have order $\alpha$, and off-block entries have order $\beta$. 
    Let $U_0$ be such that its principal symbol is $u_0$, then
    \[
    U_0^{-1}AU_0 = B_2 + R_0
    \]
    where $B_2 \in \Psi^{2,-\infty}$ has principal symbol $b_2$ and $R_0 \in \Psi^1$.

    Next, we improve the decomposition by eliminating the off-block entries of $R_1$. For $S_{-j} \in \Psi^{-j}$, we write
    \[
    U_k = U_0(I + S_{-1}+ S_{-2} + \cdots + S_{-k}),
    \]
    and require recursively that
    \[
    (U_0(I+S_{-1}+\cdots + S_{-k}))^{-1}AU_0(I + S_{-1} + \cdots + S_{-k}) \in B_2 + \Psi^{1,1-k}.
    \]
    With $S_0 = I$, the base case is done, and suppose this holds for $k$, namely
    \[
    (U_0(I+S_{-1}+\cdots + S_{-k}))^{-1}AU_0(I+S_{-1}+\cdots +S_{-k}) = B_2 + R_k,
    \]
    where $R_k \in \Psi^{1, 1-k}$. Then for $k+1$ we want
    \[
    \begin{split}
        &(U_0(I+S_{-1}+\cdots + S_{-k} + S_{-k-1}))^{-1}AU_0(I+S_{-1}+\cdots +S_{-k} + S_{-k-1})\\
        &= B_k + \Psi^{1,-k}.
    \end{split}
    \]
    Simplifying, this becomes
    \[
    (I+\cdots+S_{-k})R_k + U_0^{-1}AU_0S_{-k-1} = S_{-k-1}B_2 + \Psi^{1,-k}.
    \]
    Use the base case, we have
    \[
    (I+\cdots+S_{-k})R_k + [B_2, S_{-k-1}] = \Psi^{1,-k}
    \]
    because $R_0S_{-k-1} \in \Psi^{-k}$.
    Note that the two terms on the LHS 
    belong to $\Psi^{1,1-k}$ and $\Psi^{1-k}$, respectively, so we only need to deal with the off-block terms. We emphasize that, unlike the scalar case, $[B_2, S_{-k-1}] \in \Psi^{1-k}$ instead of $\Psi^{-k}$ because the matrices may not commute; moreover, the computation will only involve the principal symbol of each term.

    Denote $r_k, s_{-k-1}$ the corresponding principal symbols, we thus require
    \[
    r_k + b_2s_{-k-1} - s_{-k-1}b_2
    \]
    vanishes on off-block entries. Let $\alpha\beta$ be an off-block entry, recall that $b_2 = \mathrm{diag}(\lambda_1^2, \cdots, \lambda_N^2)$, then $\lambda_\alpha \neq \lambda_\beta$, and we can set
    \[
    (s_{-k-1})_{\alpha\beta} = - \frac{(r_k)_{\alpha\beta}}{\lambda_\alpha^2 - \lambda_\beta^2}.
    \]
    This finishes the construction, as we can now take
    \[
    U \sim U_0(I + S_{-1}+S_{-2}+\cdots)
    \]
    such that $U^{-1}AU$ is smoothing on the off-block entries.
\end{proof}

We now define the microlocal projection operators.
\begin{lemma}\label{lem: projection}
    Consider
    \[
    \Pi_\alpha := UP_\alpha U^{-1}
    \]
    where $P_\alpha = \mathrm{diag}(0, \cdots, I_{m_\alpha},\cdots, 0)$ and $m_\alpha$ is the multiplicity of $\lambda_\alpha^2$.
    Then
    \[
    \Pi_\alpha\Pi_\beta = \delta^\alpha_\beta\Pi_\alpha + \Psi^{-\infty}, \quad \sum_\alpha \Pi_\alpha = I + \Psi^{-\infty}, \quad
    [A, \Pi_\alpha] = \Psi^{-\infty}.
    \]
    If $A$ is $M$-self-adjoint, then
    \[
    \quad \Pi_\alpha^*M - M\Pi_\alpha = \Psi^{-\infty}.
    \]
\end{lemma}
\begin{proof}
    Throughout the proof, everything is modulo $\Psi^{-\infty}$, and we omit it for simplicity. By definition,
    \[
    \Pi_\alpha \Pi_\beta = UP_\alpha P_\beta U^{-1} = \delta^\alpha_\beta \Pi_\alpha,
    \]
    and
    \[
    \sum_\alpha \Pi_\alpha = U\sum_\alpha P_\alpha U^{-1} = UU^{-1} = I.
    \]
    Moreover,
    \[
    [A, \Pi_\alpha] = AUP_\alpha U^{-1} - UP_\alpha U^{-1}A = UBP_\alpha U^{-1} - UP_\alpha B U^{-1} = UB_jU^{-1}-UB_jU^{-1} = 0,
    \]
    where $B = \mathrm{diag}(B_1, \cdots, B_k)$ and by abuse of notation we also use $B_j = \mathrm{diag}(0, \cdots, B_j, \cdots, 0)$.
    
    Now suppose $A$ is $M$-self-adjoint. To prove the last equality, denote $G := U^*MU$, we first prove $G$ is also block-diagonal, in fact $G \in \Psi^{0,-\infty}$. Denote the principal symbol $\mathsf g_0$, then for any off-block entry $\alpha\beta$,
    \[
    (b_2^*)_{\alpha\gamma}(\mathsf g_0)_{\gamma\beta} = (\mathsf g_0)_{\alpha\gamma}(b_2)_{\gamma\beta}.
    \]
    Recall in proof of Lemma \ref{lem: decomposition} we showed the principal symbol $b_2$ of $B$ is diagonal, hence
    \[
    \lambda_\alpha^2(\mathsf g_0)_{\alpha\beta} = \lambda_\beta^2 (\mathsf g_0)_{\alpha\beta},
    \]
    which forces $(\mathsf g_0)_{\alpha\beta} = 0$ since $\lambda_\alpha^2 \neq \lambda_\beta^2$. This proves $G \in \Psi^{0,-1}$, then we improve it to $\Psi^{0,-2}$.

    Indeed, suppose $G \in \Psi^{0,-1}$ and let $(\mathsf g_{-1})_{\alpha\beta}$ be the principal symbol for the off-diagonal entry. Then the order 1 terms for $B^*G$ is
    \[
    b_1^*\mathsf g_0 + b_2^*\mathbf{g}_{-1} + \frac{1}{2i}\{b_2, \mathbf{g}_0\}.
    \]
    Evaluate at the off-block $\alpha\beta$ entry gives
    \[
    \lambda_\alpha^2(\mathsf g_{-1})_{\alpha\beta}
    \]
    because $b^*_1\mathsf g$ and $\{b_2,\mathsf g_0\}$ are both still block diagonal. Similarly one obtains the contribution from $GB$ is
    \[
    \lambda_\beta^2 (\mathsf g_{-1})_{\alpha\beta},
    \]
    and together we obtain again $(\mathsf g_{-1})_{\alpha\beta} = 0$ for all off-block entries. Repeat for lower orders, one has $G \in \Psi^{0,-\infty}$ is block diagonal as well. Thus $G$ commutes with $P_\alpha$, and we have
    \[
    \Pi_\alpha^*M = (U^{-1})^*P_\alpha U^*M = (U^{-1})^*P_\alpha G U^{-1} = (U^{-1})^*GP_\alpha U^{-1} = M\Pi_\alpha.
    \]
\end{proof}

\subsection{Mode conversion}\label{sec: mode conversion}
If locally $\partial \Omega$ is given by $x^n = 0$, we define $\pi$ the projection map
\[
\pi: T^*((0, T) \times \Omega)|_{(0, T) \times \partial \Omega} \to T^*((0, T) \times \partial \Omega), \quad (t, x', 0, \tau, \xi', \xi_n) \mapsto (t, x', \tau, \xi').
\]
For any $(t, x', \tau, \xi') \in T^*((0, T) \times \partial \Omega)$, consider
\[
\Sigma_\alpha \cap \pi^{-1}(t, x', \tau, \xi').
\]
The cardinality of the above set can only be 0, 1 or 2, and accordingly we group such $(t, x', \tau, \xi')$ and call them elliptic $\mathcal{E}_\alpha$, glancing $\mathcal{G}_\alpha$ and hyperbolic $\mathcal{H}_\alpha$ sets, respectively (see \cite[Chapter XXIV]{Hor07}).
We now briefly explain mode conversion (see Figure \ref{fig: mode conversion}), which only occurs at the boundary when there are multiple layers of characteristic sheets. The focus of this paper is the interaction of characteristic sheets in the interior at time $T$ slice instead of mode conversion on the lateral boundary, thus for more detailed explanation of mode conversion on the boundary for multi-velocity system, we refer to \cite{Tay75, HU03, HUZ22, UZ21, UZ24}.

Indeed, suppose the incoming ray is with respect to the characteristic sheet $\Sigma_\alpha$, and arrives at $(t, x)$ with corresponding covector $(\tau, \xi)$ transversally ($\xi_n \neq 0$). The polarization set is thus $E_\alpha(t, x, \tau, \xi)$. In terms of optics solution or Gaussian beam solutions, the leading amplitude $a^{\alpha, \mathrm{inc}}_0$ lies in this polarization set. The $\alpha$-reflection follows the bicharacteristic starting from $(t, x, \tau, \xi', -\xi_n)$ where we take normal coordinate near the boundary and locally $x^n = 0$ denotes the boundary. However, the leading amplitude for the reflection $\alpha$-beam has polarization set $E_\alpha(t, x, \tau, \xi', -\xi_n)$, which may not contain $a_0^{\alpha,\mathrm{inc}}$. As a result, additional $\beta$-beams may be required to cancel out the income $\alpha$-ray, at least on the principal level. Then for each characteristic sheet $\Sigma_\beta$, the situation is divided into three categories, depending on whether the projected covector $(t, x', \tau, \xi')$ lies in $\mathcal{E}_\beta$, $\mathcal{H}_\beta$ or $\mathcal{G}_\beta$.
\begin{enumerate}
    \item If in $\mathcal{E}_\beta$, then locally after the microlocal projection $\Pi_\beta$, the projected solution is smooth in the interior but has singular trace, corresponds to an \emph{evanescent wave}.
    \item If in $\mathcal{H}_\beta$, then a future pointing $\beta$-beam is generated from $(t, x, \tau, \xi^{(\beta)})$, the preimage of the projected covector intersecting $\Sigma_\beta$ with forward in time bicharacteristic, and with leading amplitude lying in $E_\beta(t, x, \tau, \xi^{(\beta)})$.
    \item If in $\mathcal{G}_\beta$, then a future pointing $\beta$-gliding beam (because the boundary is assumed to be strictly convex here, otherwise it is more complicated, see \cite[Chapter XXIV]{Hor07}) is generated from $(t, x, \tau, \xi^{(\beta)})$, the unique preimage of the projected covector intersecting $\Sigma_\beta$, and with same leading amplitude lying in $E_\beta(t, x, \tau, \xi^{(\beta)})$.
\end{enumerate}
Note that reflection can not be all evanescent waves, since the covector lies in $\mathcal{H}_\alpha$. With contribution from all characteristic sheets, the incoming $\alpha$-beam can be canceled out on the boundary, at least on the principal level. A recursive treatment can then cancel out the incoming $\alpha$-beam on the boundary up to smoothing error.

Since we will only work with directions that are away from the glancing sets, let us now describe in more details the mode conversion for these directions, where only evanescent waves and reflection waves can happen. Suppose $\Gamma \subset \mathcal{H}_\alpha$ is the conic Lagrangian submanifold corresponds to an incoming Lagrangian distribution $u^{\alpha, \mathrm{inc}}|_{x^n=0}$ generated from some conormal distribution Dirichlet data. For the set of $\beta$ such that $\Gamma \subset \mathcal{H}_\beta$, denote $\Gamma^{\beta, +} \sqcup \Gamma^{\beta, -} = \pi^{-1}\Gamma \cap \Sigma_\beta$ the outward and inward pointing lightlike preimage. Also denote $\mathcal{F}^{\beta, \pm}$ the extension of $\Gamma^{\beta, \pm}$ into the interior along $\beta$-bicharacteristics. Locally, the incoming Lagrangian distribution has the form
\[
u^{\alpha, \mathrm{inc}}(t, x) = \int e^{i\phi^{\alpha, \mathrm{inc}}(t, x, \tau, \xi)} a^{\alpha, \mathrm{inc}}(t, x, \tau, \xi) d\tau d\xi
\]
such that
\[
\{(t, x, d_t\phi^{\alpha, \mathrm{inc}}, d_x\phi^{\alpha, \mathrm{inc}}): d_{\tau, \xi}\phi^{\alpha, \mathrm{inc}} = 0\}
\]
parametrizes $\mathcal{F}^{\alpha, -}$ and
\[
a^{\alpha, \mathrm{inc}} \sim a_0^{\alpha, \mathrm{inc}} + a_1^{\alpha, \mathrm{inc}} + \cdots,
\]
is a classical symbol. Then the reflection wave has the form
\[
u^\mathrm{ref} = \sum_{\beta: \Gamma \subset \mathcal{H}_\beta}\int e^{i\phi^{\beta, \mathrm{ref}}}a^{\beta,\mathrm{ref}} + \sum_{\beta: \Gamma \subset \mathcal{E}_\beta} \int e^{i\phi^{\beta, \mathrm{eva}}}a^{\beta,\mathrm{eva}}
\]
where the symbols are also classical.

For every $(t, x', \tau, \xi') \in \Gamma$, solve $\xi_n$ from
\[
\det (\tau^2-a_2(x',0,\xi', \xi_n)) = 0.
\]
Note that $\tau^2-a_2$ is the product of $k$ equations each corresponds to one of the $\beta$-characteristics.
For each $\beta$ such that $\Gamma \subset \mathcal{H}_\beta$, $\pi^{-1}(t, x', \tau, \xi') \cap \Sigma_\beta$ provides two solutions, with one and only one of them corresponds to forward in time bicharacteristic in the interior, denote it as $\xi^\beta_n$. For each $\beta$ such that $\Gamma \subset \mathcal{E}_\beta$, the equation corresponds to it has no real solution, and we choose the complex root that corresponds to exponential decaying into the interior, also 
denote it as $\xi^\beta_n$. Then we require
\[
\phi^{\beta, *}|_{x^n=0} = \phi^{\alpha, \mathrm{inc}}|_{x^n=0}, \quad \partial_n\phi^{\beta, *}|_{(t, x')} = \xi_n^\beta
\]
for $* = \mathrm{ref,eva}$.

Let $E_\beta$ be the kernel of $\tau^2-a_2(x',0,\xi',\xi_n^\beta)$, note that it is precisely the polarization set for those $\beta$ such that $\Gamma \subset \mathcal{H}_\beta$. Then we have
\[
\bigoplus_\beta E_\beta \simeq \mathbb{C}^N,
\]
this follows for non-glancing directions from the uniform Lopatinski condition, as shown in \cite[Section 5]{Tay75} (where the condition is called the Kreiss condition). For the later examples, namely multi-velocity wave system and Lam\'e system, this can be explicitly computed. Indeed, for multi-velocity wave system the polarization sets are location and direction independent, so this always holds; for Lam\'e system, see for example \cite[Section 3]{UZ24}.

Denote $\mathcal{P}_\beta$ the unique decomposition of $\mathbb C^N$ according to $E_\beta$, then for $(t, x', \tau, \xi') \in \Gamma$, the principal symbol of the reflection wave on $\Gamma^{\beta, +}$ is given by
\[
a_0^{\beta, \mathrm{ref}} = -\mathcal{P}_\beta a_0^{\alpha, \mathrm{inc}}.
\]
Follow the standard optics construction, by applying the operator to the $\beta$-reflection wave and matching degree, we have the phase function in the interior solves the Eikonel equation, while the principal amplitude solves a transport equation along the $\beta$-Hamiltonian flow and lives in the polarization set. The subprincipal amplitude $a_1^{\beta, \mathrm{ref}}$, on the other hand, does not necessarily live in the polarization set. In particular, it can be decomposed into the sum of two parts: $a_{1, \parallel}^{\beta, \mathrm{ref}} + a_{1, \perp}^{\beta, \mathrm{ref}}$, where the parallel part lives in the polarization set and the orthogonal part can be solved algebraically from the leading amplitude. The parallel part on $\Gamma^{\beta, +}$ is again given by $\mathcal{P}_\beta$ and $a_1^{\alpha, \mathrm{inc}}$, and is solved in the interior along the transport equation along $\beta$-Hamiltonian flows. Recursively solving the lower order amplitudes, the reflection waves are determined.
For this recursive construction of optic solutions in the interior, we refer to Section \ref{sec: beam solutions} and Section \ref{sec: lame beam solutions} for detailed computations in the multi-velocity wave system and Lam\'e system as examples.

The principal symbol of the evanescent wave on $\Gamma$ is given by
\[
a_0^{\beta, \mathrm{eva}} = -\mathcal{P}_\beta a_0^{\alpha, \mathrm{inc}}.
\]
In particular, the evanescent wave solves a parabolic type equation with $x^n$ treated as time moving into the interior, and is thus smooth in the interior. We refer to \cite{Tay75} for a detailed discussion, and also Section 1 of \cite{LU89} where the authors deal with the symbol in the elliptic region by solving a backward heat equation. For microlocal analysis the evanescent waves are thus part of the smooth remainder in the interior.

Therefore, for an incoming Lagrangian distribution associated with $\mathcal{F}^{\alpha, -}$ such that the projection on the boundary $\Gamma$ stays away from the glancing regions, the outgoing solution can be decomposed into $\beta$-reflection waves that are outgoing Lagrangian distributions associated with $\mathcal{F}^{\beta, -}$ for $\beta$ such that $\Gamma \subset \mathcal{H}_\beta$, and $\beta$-evanescent waves that is smooth in the interior for $\beta$ such that $\Gamma \subset \mathcal{E}_\beta$.

\begin{figure}[htbp]
    \centering

    \begin{subfigure}{0.32\textwidth}
        \centering
        \includegraphics[width=\linewidth]{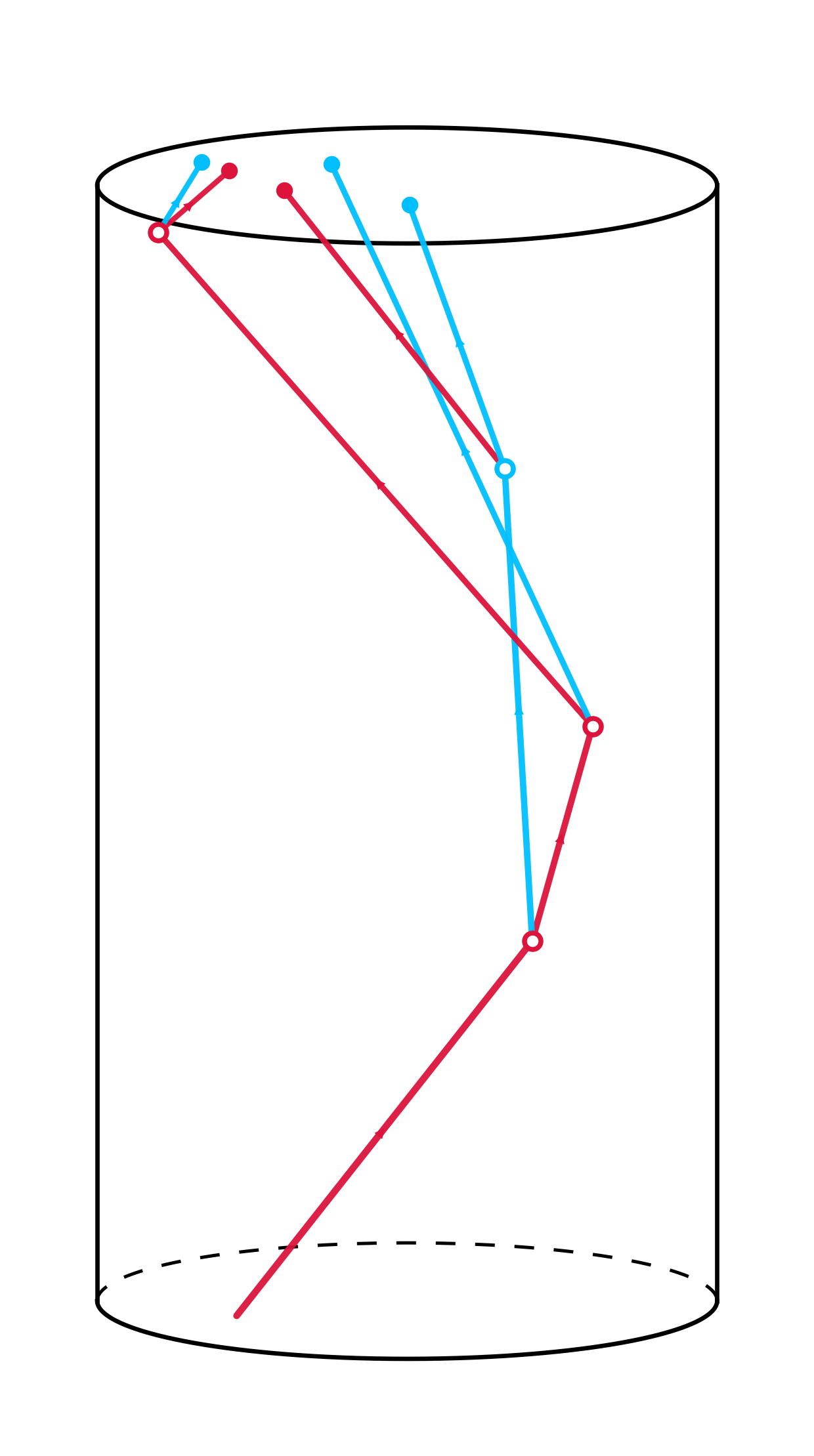}
        \caption{}
        \label{fig:a}
    \end{subfigure}
    \hfill
    \begin{subfigure}{0.32\textwidth}
        \centering
        \includegraphics[width=\linewidth]{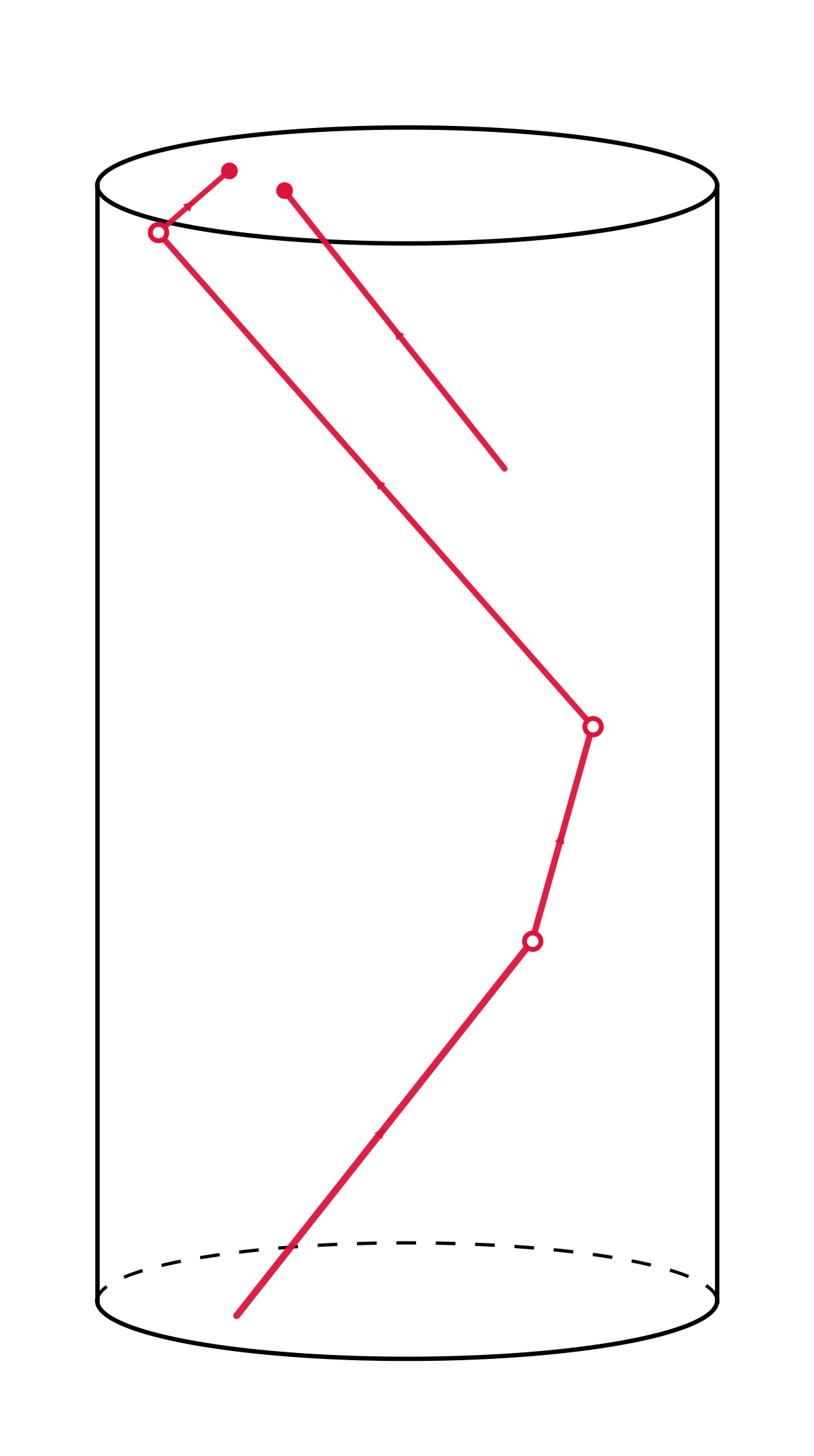}
        \caption{}
        \label{fig:b}
    \end{subfigure}
    \hfill
    \begin{subfigure}{0.32\textwidth}
        \centering
        \includegraphics[width=\linewidth]{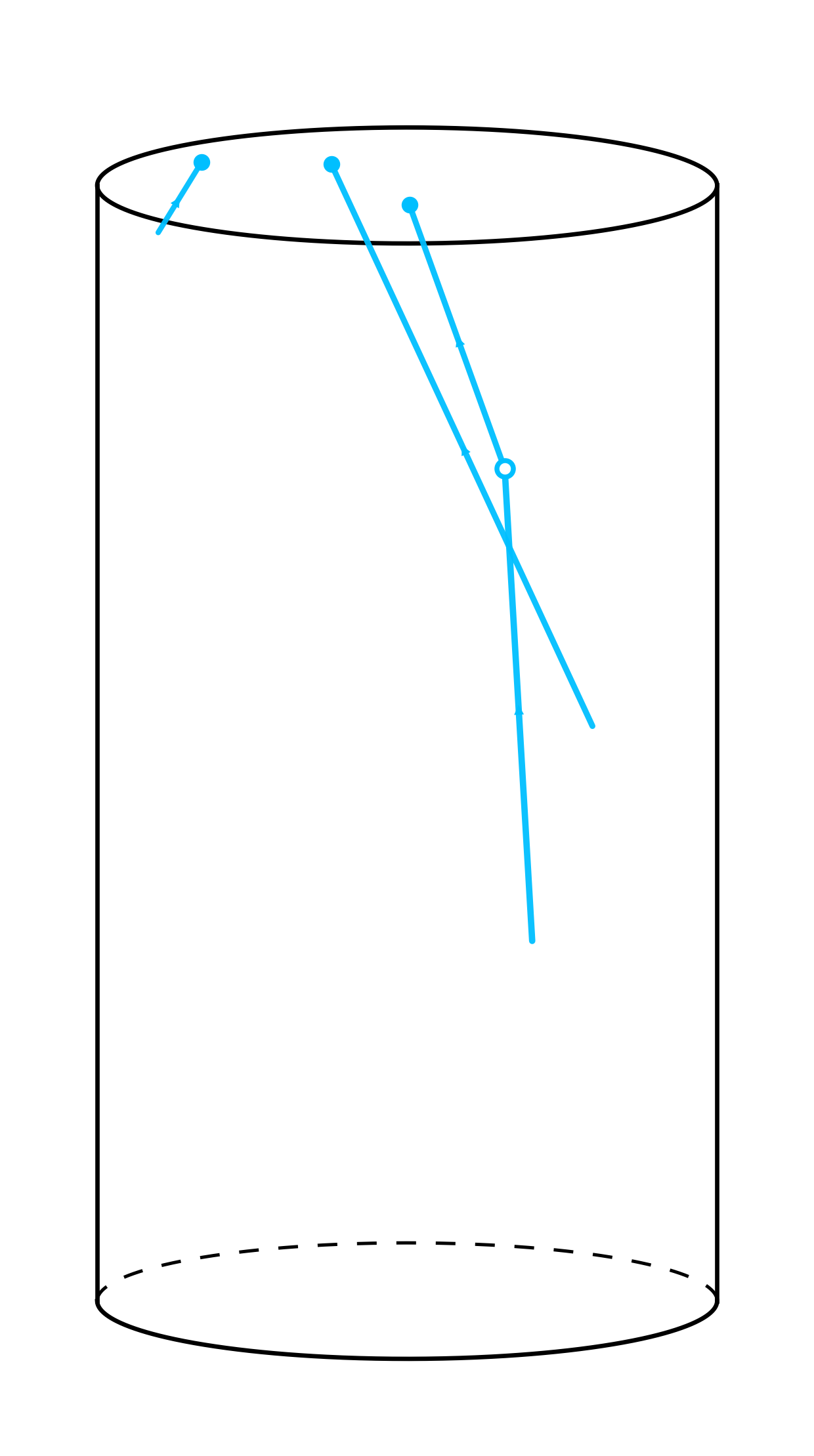}
        \caption{}
        \label{fig:c}
    \end{subfigure}

    \caption{When the metrics are smooth, mode conversion only occurs on the lateral boundary. In $(A)$, the singularity is set such that the initial propagation only occurs along red beams. However, when the beam reaches the boundary, the reflection may contain both the read beam according to Snell's Law and the blue beam from mode conversion. This may happen (or not) for every reflection for both red and blue beams, resulting in a tree-like structure for the solution $u^f$. A microlocal decoupling can separate out the propagation along each characteristic sheets, thus $(B)$ and $(C)$ corresponds to $\Pi_\alpha u^f$, namely the solution microlocally projected onto the red and blue characteristic sheets, respectively. Thus, even though the boundary Dirichlet data for the solution $u^f$ may only be singular at a single location, the boundary Dirichlet data for the microlocally projected solution $\Pi_\alpha u^f$ may be singular at multiple locations due to mode conversion. Mode conversion is different from the cross-sheet interaction we study in this paper, which is the interaction between the modes at time $T$. In other words, we focus on whether the red and blue beam at time $T$ slice interact with each other or not based on the choice of $K$.}
    \label{fig: mode conversion}
\end{figure}

Now let us introduce some notations to describe the flows along which the singularities flow.
\begin{definition}
    A generalized broken bicharacteristic is a piecewise continuous curve in $T^*((0, T)\times \Omega)$ such that
    \begin{enumerate}
        \item it flows strictly monotone in time;
        \item every maximally connected component in the interior is a bicharacteristic with respect to $\Sigma_\alpha$ for some $\alpha$;
        \item every maximally connected component in the boundary is an integral curve of glancing vector field with respect to $\Sigma_\beta$ for some $\beta$;
        \item the location of discontinuity only occurs at the boundary, and the projection under $\pi$ for the two ends are equal to each other.
    \end{enumerate}
    An $\alpha$-generalized broken bicharacteristic is a generalized broken bicharacteristic that completely lies in $\Sigma_\alpha$, namely without mode conversion. A ($\alpha$-)broken bicharacteristic is a ($\alpha$-)generalized broken bicharacteristic that is piecewise ($\alpha$-)bicharacteristic and completely lies in the interior except end points, namely without glancing rays.
\end{definition}
Note that because of the strictly convexity assumption, the glancing behavior is simplified greatly, and we refer to \cite{Hor07} for a detailed definition of bicharacteristic and glancing vector field.
\begin{definition}
    For every $(y, \eta) \in T^*\Omega^\circ \backslash 0$, we denote $C_\alpha(y, \eta)$ the set of $(t, x', \tau, \xi') \in T^*((0, T) \times \partial \Omega)$ such that there exists a generalized broken bicharacteristic from $\pi^{-1}(t, x', \tau, \xi') \cap \Sigma_\beta$ for some $\beta = 1, \cdots, k$ that ends at $(T, y, \tau_\alpha, \eta) \in \Sigma_\alpha$. We define the relation
    \[
    \mathcal{C}_\alpha = \{(y, \eta; t, x', \tau, \xi'): (t, x', \tau, \xi') \in C_\alpha(y, \eta)\}.
    \]
\end{definition}
If $(y, \eta; t, x', \tau, \xi') \in \mathcal{C}_\alpha$ is such that all the generalized broken bicharacteristic from preimage of $(t, x', \tau, \xi')$ to $(T, y, \tau_\alpha, \eta)$ are piecewise bicharacteristics (namely the reflection points are away from glancing sets), then locally it is the image of a canonical map since it is obtained via the composition of several Hamiltonian flows, see \cite{MS78, MS82}. Thus to stay away from the glancing directions, we denote
\[
\mathcal{H} \subset \cup_\alpha \mathcal{H}_\alpha
\]
the set of directions whose corresponding generalized broken bicharacteristics are all piecewise bicharacteristics.
We also denote
\[
\mathcal{C}_\alpha^t := \{(t, x', \tau, \xi'; y, \eta): (y, \eta; t, x', \tau, \xi') \in \mathcal{C}_\alpha\}.
\]

\subsection{Microlocal property of normal operator}
\begin{lemma}\label{lem: W^T_alpha FIO}
    Assume \textbf{(H)} holds. Denote $Q \in \Psi^0((0, T) \times \partial \Omega; \mathrm{End}(E))$ a microlocal cutoff near some $(t, x', \tau, \xi') \in \mathcal{H}$. Denote $W^T_\alpha = \Pi_\alpha W^T$, then $W^T_\alpha Q$ is a finite sum of FIOs, each canonical relation is given by a canonical graph contained in $\mathcal{C}_\alpha$ restricted to $\mathrm{WF}(Q)$, namely
    \[
    \mathcal{C}_\alpha \cap T^*\Omega^\circ \times \mathrm{WF}(Q).
    \]
\end{lemma}
\begin{proof}
    Let $u^f$ be the solution with Dirichlet data $Qf$. Since $\mathrm{WF}(Q) \subset \mathcal{H}$, all the generalized broken bicharacteristics are piecewise bicharacteristics correspond to reflections and mode conversions. In particular, the wavefront set of $u^f$ can be written as
    \[
    \begin{split}
        \mathrm{WF}(u^f) \subset \bigcup_{\beta = 1}^k \{&\gamma \text{ is a $\beta$-broken bicharacteristic such that}\\
        &\text{$\gamma \subset \tilde \gamma$ a generalized broken bicharacteristic starting from $\mathrm{WF}(Qf)$}\}.
    \end{split}
    \]
    In other words, the singularity of $u^f$ flows along a $\beta$-broken bicharacteristic $\gamma$ only when $\gamma$ is originated from $\mathrm{WF}(Qf)$ directly or generated from mode conversion of some generalized broken bicharacteristic $\tilde \gamma$ originated from $\mathrm{WF}(Qf)$.
    
    Denote $v^f = U^{-1}u^f$, then it solves
    \[
    \partial_t^2 v^f + Bv^f + Rv^f = 0
    \]
    for some smoothing operator $R \in \Psi^{-\infty}$. In particular,
    \[
    W^T_\alpha Qf = UP_\alpha U^{-1}W^TQf = UP_\alpha v^f|_{t=T}.
    \]
    Since $B = \mathrm{diag}(B_1, \cdots, B_k)$ is block diagonal, we have
    \[
    \partial_t^2 v^f_\alpha + B_\alpha v^f_\alpha + R_\alpha v^f = 0
    \]
    where $v^f_\alpha = P_\alpha v^f$ and $R_\alpha = P_\alpha R$ is smoothing. In fact, since $B_\alpha$ has principal symbol $\lambda_\alpha^2$ (see proof of Lemma \ref{lem: decomposition}), by the assumption that $\lambda_\alpha^2$ is quadratic polynomial, the leading order operator for the above equation of $v^f_\alpha$ agrees with wave equation $\partial_t^2 - \Delta_{g_\alpha}$. In particular, the system has only one characteristic sheet $\Sigma_\alpha$, then propagation can only occur with respect to its broken bicharacteristics, and the wavefront set of $Uv^f_\alpha$ in the interior at any time slice is a subset of $\Sigma_\alpha$, see Figure \ref{fig:b} and \ref{fig:c}.
    
    Thus, the wavefront set of $v^f_\alpha$ is
    \[
    \mathrm{WF}(v^f_\alpha) \subset \mathrm{WF}(u^f) \cap \Sigma_\alpha.
    \]
    In other words, the singularity of $v^f_\alpha$ can only flow along $\alpha$-broken bicharacteristic $\gamma$ where $\gamma$ is originated from $\mathrm{WF}(Qf)$ directly or generated from mode conversion of some generalized broken bicharacteristic from $\mathrm{WF}(Qf)$. As a result,
    \[
    \mathrm{WF}(Uv^f_\alpha|_{t = T}) \subset \mathcal{C}_\alpha \circ \mathrm{WF}(Qf).
    \]
    The relation restricted to $T^*\Omega^\circ \times \mathrm{WF}(Q)$ is a canonical relation because all the generalized broken bicharacteristics from $\mathrm{WF}(Q)$ are piecewise bicharacteristics away from any glancing region, which is essentially the composition of scattering relations, see \cite{MS78, MS82}. To show the operator is a finite sum of FIOs, one can follow the distorted plane wave construction (see \cite{KLU18, HUZ22}) by setting up a source term in the exterior region whose local solution matches the boundary data up to smoothing error, and the reflection is constructed in previous section as the sum of reflection waves that are still Lagrangian distributions and smooth evanescent waves. Then $v^f_\alpha$ is a finite sum of such constructions, where each construction corresponds to either the initial boundary Dirichlet data, or new singularity generated from $\alpha$-mode conversion. Note that each construction essentially corresponds to a canonical graph. Two such canonical graphs may even have the same domain and range, this corresponds to two sets of different generalized broken bicharacteristics with the same starting and ending point. However, as they do not experience the same reflection and mode conversion pattern during the procedure, the canonical graphs are in general not the same. Thus it is better to describe $W^T_\alpha Q$ as the sum of FIOs rather than a single FIO, since the locally finite sum of oscillatory integrals may contain phase functions that parametrizes two different Lagrangian submanifolds whose left and right projections coincide as open sets.


\end{proof}

\begin{theorem}\label{thm: microlocal}
    Assume \textbf{(H)} holds. Denote $Q \in \Psi^0((0, T) \times \partial \Omega; \mathrm{End}(E))$ a microlocal cutoff with $\mathrm{WF}(Q) \subset \mathcal{H}$, and $K \in C^\infty(\Omega; \mathrm{End}(E))$ is positive definite.
    Then $Q^*C^T_KQ$ is a locally finite sum of FIOs with canonical relation contained in
    \[
    \bigcup_{\alpha, \beta} \mathcal{C}_\beta^t \circ \mathcal{C}_\alpha.
    \]
    If $A$ is $M$-self-adjoint with respect to some positive-definite $M \in C^\infty(\Omega; \mathrm{End}(E))$, then the canonical relation for $Q^* C^T_M Q$ reduces to subset of
    \[
    \bigcup_\alpha \mathcal{C}_\alpha^t \circ \mathcal{C}_\alpha.
    \]
\end{theorem}

\begin{remark}\label{rmk: self adjoint too strong}
    In fact, a weaker condition will cause the vanishing of cross-sheet interaction. In the proof, this is caused by
    \[
    \Pi^*_\alpha M \Pi_\beta = \Psi^{-\infty}
    \]
    when $\alpha \neq \beta$. For $M$-self-adjoint $A$, this is deduced from
    \[
    \Pi_\alpha\Pi_\beta = \delta^\alpha_\beta \Pi_\alpha + \Psi^{-\infty}, \quad \Pi_\alpha^*M = M\Pi_\alpha + \Psi^{-\infty}
    \]
    in Lemma \ref{lem: projection}. A weaker condition is simply directly asking
    \[
    \Pi^*_\alpha M \Pi_\beta = \Psi^{-\infty}, \quad \alpha \neq \beta.
    \]
    This includes $M$ that does not make the operator self-adjoint, but is block diagonal with respect to the microlocal characteristic splitting. For a concrete example, see Section \ref{sec: cancellation wave} and Remark \ref{rmk: example self adjoint too strong}. We choose to work with $M$-self-adjoint case because it is more intuitive.
\end{remark}

\begin{proof}
    By Lemma \ref{lem: decomposition}, there exists elliptic $U \in \Psi^0$ such that $U^{-1}AU = B \in \Psi^{2,-\infty}$.
    Denote $\Pi_\alpha$ the projection operator from Lemma \ref{lem: projection}. Again we omit all the $\Psi^{-\infty}$ terms. Recall $W_\alpha^T = \Pi_\alpha W^T$, then
    \[
    C^T_K = (W^T)^*KW^T = \sum_{\alpha\beta}(W^T_\beta)^*KW^T_\alpha.
    \]
    By Lemma \ref{lem: W^T_alpha FIO}, $W^T_\alpha Q$ is a locally finite sum of FIOs where each canonical relation is a canonical graph and subset of $\mathcal{C}_\alpha$, essentially connected only by generalized broken bicharacteristics that are piecewise bicharacteristics in the interior. Since each of the canonical relation is given by a canonical graph, composition is well-defined. As a result, the composition is still a locally finite sum of FIOs with canonical relation sits inside
    \[
    \mathcal{C}_\beta^t \circ \mathcal{C}_\alpha.
    \]
    Summing up all $\alpha$, $Q^*C^T_KQ$ is thus a locally finite sum of FIOs with canonical relation in
    \[
    \bigcup_{\alpha , \beta} \mathcal{C}_\beta^t \circ \mathcal{C}_\alpha.
    \]
    

    Now assume $A$ is $M$-self-adjoint for some positive definite $M$. Then
    \[
    C^T_M = \sum_{\alpha\beta}(W^T_\beta)^*MW^T_\alpha = \sum_\alpha (W^T_\alpha)^*MW^T_\alpha,
    \]
    where we used $\sum_\alpha \Pi_\alpha = I$, $\Pi_\alpha^*M = M\Pi_\alpha$, and $\Pi_\alpha \Pi_\beta = 0$ for $\alpha \neq \beta$ from Lemma \ref{lem: projection}. Thus the canonical relation for $Q^*C^T_MQ$ reduces to
    \[
    \bigcup_\alpha \mathcal{C}_\alpha^t \circ \mathcal{C}_\alpha.
    \]
\end{proof}

\section{Toy model: leakage detection for multi-velocity wave system}\label{sec: toy model}

We have shown that the cross-sheet interaction is not visible for $C^T_M$ when the operator $A$ is $M$-self-adjoint. For other cases, there is possibility of observing this cross-sheet interaction, and we show that this can provide much richer geometric information. Specifically, we will use Gaussian beam solutions to compute how this interaction can be observed, and how to use this interaction data.

In this section, we demonstrate off-polarization leakage detection with a simple two-velocity wave system on $[0,T]\times \Omega$, where $\Omega\subset \mathbb R^n$:
\begin{equation}\label{eq: 2 wave equation}
    L_{P,S}u := \begin{pmatrix}
        \Box_P & 0 \\
        0 & \Box_S
    \end{pmatrix}u + \sum_{\alpha=1}^n A^\alpha(x) \partial_\alpha u + V(x) u = 0,
\end{equation}
and for $* = P, S$, $\Box_*$ is the wave operator with respect to $c_*^{-2}(x)$:
\[
\Box_* = -\partial_t^2 + \nabla \cdot (c_*^2(x)\nabla).
\]
One may use the Laplace-Beltrami operator with respect to $g_* = c_*^{-2}(x)dx^2$, which only changes the first order terms $A^\alpha$ on the diagonal part, it will not affect the results we have. We assume that
\[
c_P > c_S > 0
\]
at any point in $\Omega$, this is to avoid the overlapping of characteristic sheets. 
The principal symbol is a diagonal matrix $\mathrm{diag}(p_P, p_S)$ with
\[
p_*(t, x, \tau, \xi) = \tau^2 - c_*^2|\xi|^2, \quad * = P,S.
\]
We also denote
\[
A^\alpha = \begin{pmatrix}
    A^\alpha_{PP} & A^\alpha_{PS} \\
    A^\alpha_{SP} & A^\alpha_{SS}
\end{pmatrix}.
\]
The reason why we choose this model to demonstrate is that its polarization sets are clear, with $P$-polarization set being $(1, 0)$ and $S$-polarization set being $(0, 1)$ for any point and direction $(x, \xi)$. For a detailed introduction of polarization sets, see \cite{Den82, HU03}. For simplicity, we shall call the first component the $P$-component, and the second component the $S$-component.

Denote $W^T: f \mapsto u^f(T, \cdot)$ where $u^f$ solves \eqref{eq: 2 wave equation} with vanishing initial data and boundary Dirichlet data $f$ supported in $(0, T)\times \partial \Omega$. The normal operator is given by
\[
C^T_I := (W^T)^*W^T,
\]
where $(W^T)^*$ is the $L^2(dx)$ adjoint and $I$ is the identity matrix.
We also study the more general $k$-velocity system
\begin{equation}\label{eq: k wave equation}
    L_ku := \mathrm{diag}(\Box_{P_1} ,\cdots, \Box_{P_k})u + \sum_{\alpha=1}^n A^\alpha(x) \partial_\alpha u + V(x)u = 0,
\end{equation}
where the velocities satisfy $c_1 > c_2 > \cdots > c_k > 0$, and $A^\alpha = (A^\alpha_{ij})$.

\subsection{Gaussian beams and off-polarization leakage}\label{sec: beam solutions}

We now compute the Gaussian beam solutions and the equations they satisfy. We refer to \cite{KKL01, Ral82} for details about Gaussian beams, optics solutions, and WKB solutions; also see \cite{JJ15, CP83, UZ21} for Gaussian beams in multi-velocity systems. To compute the off-polarization leakage terms, we first recall the local Gaussian beam construction for the model system.  The discussion in this subsection is local near a fixed bicharacteristic, and all equations are understood in the usual finite-order Gaussian beam sense along the central ray.

For $*=P,S$, an $*$-Gaussian beam is an approximate solution of the form
\begin{equation}\label{eq: gaussian beam}
u_h^*(t,x)
=
h^{-n/4}\chi(t,x)e^{i\phi_*(t,x)/h}
\sum_{r=0}^N h^r a_r^*(t,x).
\end{equation}
Here $\chi$ is supported in a small tubular neighborhood of the central ray, the phase $\phi_*$ is complex-valued with
\[
\operatorname{Im}\phi_*(t,x)\geq c\,\operatorname{dist}(x,x_*(t))^2
\]
near the ray, and the amplitudes $a_r^*$ are vector-valued. The factor $h^{-n/4}$ is the standard $L^2$-normalization on a time slice, since the beam is concentrated at spatial scale $h^{1/2}$. It plays no role in the transport equations below.

The phase $\phi_*$ is chosen so that it solves the $*$-eikonal equation to sufficiently high order along the central ray:
\[
p_*(t,x,\partial_t\phi_*,\nabla_x\phi_*)=0.
\]
Equivalently, if $\zeta_*= (\tau_*, \xi_*) = (\partial_t\phi_*,\nabla_x\phi_*)$, then $p_*(\zeta_*)=0$ along the ray, while $p_P(\zeta_S) \neq 0$ and $p_S(\zeta_P) \neq 0$ because the two characteristic sheets are disjoint. The first-order Taylor coefficient of the phase gives the central covector, and the quadratic Taylor coefficient is chosen with positive imaginary part. The usual Riccati equation determines the quadratic part, while higher Taylor coefficients may be chosen recursively. Since we only need a finite-order approximate solution, all these equations are solved only up to a prescribed finite order along the central ray.

We now record the conjugated operator. For a scalar amplitude $a$,
\[
h^2e^{-i\phi/h}\Box_*(e^{i\phi/h}a)
=
p_*(d\phi)a
+
hi\,T_*[\phi]a
+
h^2\Box_* a,
\]
where
\[
T_*[\phi]a
=
-2\phi_t\partial_ta
+
2c_*^2\nabla\phi\cdot\nabla a
+
(\Box_*\phi)a.
\]
The first-order coupling terms give
\[
h^2e^{-i\phi/h}
A^\alpha\partial_\alpha(e^{i\phi/h}a)
=
hi\,A^\alpha(\partial_\alpha\phi)a
+
h^2A^\alpha\partial_\alpha a.
\]
Thus, writing
\[
A(\nabla\phi):=\sum_{\alpha=1}^n A^\alpha\partial_\alpha\phi,
\]
we obtain
\[
h^2e^{-i\phi/h}L_{P,S}(e^{i\phi/h}a)
=
P(d\phi)a
+
hi\left(
\begin{pmatrix}
T_P[\phi]&0\\
0&T_S[\phi]
\end{pmatrix}a
+
A(\nabla\phi)a
\right)
+
h^2R_\phi a,
\]
where $R_\phi$ contains the remaining second-order terms acting on the amplitude, the derivatives $A^\alpha\partial_\alpha a$, and the zero-order term $Va$.

Substituting
\[
a\sim \sum_{r=0}^N h^r a_r
\]
gives the recursive equations
\[
P(d\phi)a_0=0,
\]
and, for $r\geq 0$,
\[
P(d\phi)a_{r+1}
+
i\left(
\begin{pmatrix}
T_P[\phi]&0\\
0&T_S[\phi]
\end{pmatrix}a_r
+
A(\nabla\phi)a_r
\right)
+
R_\phi a_{r-1}
=0,
\]
with the convention $a_{-1}=0$. These equations are imposed along the central ray, together with the corresponding transverse Taylor equations, to the finite order required by the construction.

For a $P$-beam, the leading amplitude is chosen in the $P$-polarization space:
\[
a_0^P=
\begin{pmatrix}
b\\
0
\end{pmatrix}.
\]
Since $p_P(d\phi_P)=0$ and $p_S(d\phi_P)\neq 0$, the equation
\[
P(d\phi_P)a_0^P=0
\]
forces the leading $S$-component to vanish. The order $h$ equation then gives a transport equation for $b$ in the $P$-component and an algebraic equation for the first $S$-component of $a_1^P$. This algebraic component is the first off-polarization leakage of a $P$-beam.

Similarly, for an $S$-beam,
\[
a_0^S=
\begin{pmatrix}
0\\
d
\end{pmatrix},
\]
and the equation $P(d\phi_S)a_0^S=0$ forces the leading $P$-component to vanish. The order $h$ equation gives the transport equation for $d$ and an algebraic equation for the first $P$-component of $a_1^S$.


\subsubsection{P-beams}


We now compute the local $P$-beam construction near some point of interest up to the first off-polarization leakage. For notation simplicity, we omit the cutoff function and write asymptotically
\[
u_h^P \sim e^{i\phi_P/h}(a_0^P+h a_1^P+h^2a_2^P+\cdots),
\qquad 
a_j^P=\binom{(a_j^P)_P}{(a_j^P)_S},
\]
with the understanding that the Gaussian beam solutions have the rigorous form \eqref{eq: gaussian beam}.
As explained above, we choose the leading amplitude to be purely $P$-polarized:
\[
a_0^P=\binom{b}{0}.
\]
The phase $\phi_P$ is chosen to solve the $P$-eikonal equation
\[
p_P(t,x,\partial_t\phi_P,\nabla_x\phi_P)=0,
\]
where $\zeta_P=(\partial_t\phi_P,\nabla_x\phi_P)$. We have the
standard conjugation formula
\[
h^2e^{-i\phi/h}\square_*\big(e^{i\phi/h}a\big)
=
p_*(d\phi)a
+h\, i\mathcal T_* a
+h^2\square_*a,
\]
where
\[
\mathcal T_*a
=
-2\partial_t\phi\,\partial_ta
+2c_*^2\nabla_x\phi\cdot\nabla_xa
+
(\square_*\phi)a.
\]
Similarly,
\[
h^2e^{-i\phi/h}A^\alpha\partial_\alpha
\big(e^{i\phi/h}a\big)
=
h\, i A^\alpha(\partial_\alpha\phi)a
+h^2 A^\alpha\partial_\alpha a.
\]
Therefore, after multiplying the equation by $h^2$, the coefficient of
order $h^0$ is
\[
\begin{pmatrix}
p_P(\zeta_P)&0\\
0&p_S(\zeta_P)
\end{pmatrix}a_0^P=0.
\]
Since $p_P(\zeta_P)=0$, this equation imposes no condition on the
$P$-component $b$. On the other hand, because $p_S(\zeta_P)\neq 0$
on the $P$-characteristic set, it forces
\[
(a_0^P)_S=0.
\]
This is consistent with our choice $a_0^P=(b,0)^T$.

At order $h^1$, we obtain
\[
\begin{pmatrix}
p_P(\zeta_P)&0\\
0&p_S(\zeta_P)
\end{pmatrix}a_1^P
+
i
\begin{pmatrix}
\mathcal T_P&0\\
0&\mathcal T_S
\end{pmatrix}a_0^P
+
i\sum_{\alpha=1}^n A^\alpha \xi_{P,\alpha} a_0^P
=0.
\]
Now we substitute $a_0^P$, the $P$-component equation becomes the usual transport equation for the
principal $P$-amplitude:
\[
\mathcal T_P b
+
\sum_{\alpha=1}^n A^\alpha_{PP}\xi_{P,\alpha} b
=0.
\]
Here we used $p_P(\zeta_P)=0$, so the term involving $(a_1^P)_P$
drops out. This equation determines $b$ along the $P$-bicharacteristic.

For the $S$-component, the same substitution gives
\[
p_S(\zeta_P)(a_1^P)_S
+
i\sum_{\alpha=1}^n A^\alpha_{SP}\xi_{P,\alpha} b
=0.
\]
There is no term $\mathcal T_S b$, because the diagonal $S$-wave
transport operator acts on the $S$-component of $a_0^P$, and this component is zero. Thus the first $S$-component is determined algebraically:
\[
(a_1^P)_S
=
-i\frac{\sum_{\alpha=1}^n A^\alpha_{SP}\xi_{P,\alpha}}{p_S(\zeta_P)}
\,b.
\]
On the $P$-characteristic set, we have
\[
p_S(\zeta_P)
=
\tau_P^2-c_S^2|\xi_P|^2
=
(c_P^2-c_S^2)|\xi_P|^2.
\]
Therefore, the first off-polarization leakage term is
\[
(a_1^P)_S
=
-i\frac{\sum_{\alpha=1}^n A^\alpha_{SP}\zeta_{P,\alpha}}
{(c_P^2-c_S^2)|\xi_P|^2}\,b.
\]


\subsubsection{S-beams}
Similarly, for an $S$-beam
\[
u_h^S \sim e^{i\phi_S/h}\left(a_0^S+h a_1^S+\cdots\right),
\qquad
a_0^S=
\begin{pmatrix}
0\\d
\end{pmatrix},
\]
with
\[
\zeta_S=(\tau_S,\xi_S)
=
(\partial_t\phi_S,\nabla_x\phi_S),
\qquad
p_S(\zeta_S)=0,
\]
one gets
\[
p_P(\zeta_S)(a_1^S)_P
+
i\sum_{\alpha=1}^n A^\alpha_{PS}\xi_{S,\alpha}\, d
=
0.
\]
Thus
\[
(a_1^S)_P
=
-i
\frac{\sum_{\alpha=1}^n A^\alpha_{PS}\xi_{S,\alpha}}
{p_P(\zeta_S)}
\,d .
\]
On $\{p_S=0\}$,
\[
p_P(\zeta_S)
=
\tau_S^2-c_P^2|\xi_S|^2
=
-(c_P^2-c_S^2)|\xi_S|^2.
\]
Hence, we have
\[
(a_1^S)_P
=
i
\frac{\sum_{\alpha=1}^n A^\alpha_{PS}\xi_{S,\alpha}}
{(c_P^2 - c_S^2) |\xi_S|^2}
\,d .
\]



\subsubsection{Off-polarization leakage detection}\label{sec: off-polarization leakage detection}

We are now ready for off-polarization leakage detection. Denote
\[
H_* = c_*^2\xi \cdot \nabla_x - \frac{1}{2}|\xi|^2\nabla c^2_* \cdot \nabla_\xi
\]
the Hamiltonian vector field on $T^*\Omega$ with respect to $g_* = c^{-2}_*(x)dx^2$, for $* = P,S$. We denote $\varphi_*(x, \xi, \ell)$ the corresponding integral curve but parametrized with arc-length $\ell$, namely $\varphi_*(x, \xi, \ell)$ is the covector on the integral curve from $(x, \xi)$ such that the segment from $(x, \xi)$ to $\varphi_*(x, \xi, \ell)$ projected onto $\Omega$ has length $\ell$ with respect to $g_*$.

\begin{theorem}\label{thm: leakage detection}
    For $* = P,S$, let $x_* \in \partial \Omega$, $\xi_* \in T^*_{x_*}\Omega$ inward-pointing, $\tau_* = c_*|\xi_*|$, and $t_* \in (0, T)$ such that $(x^T_*, \xi^T_*):=\varphi_*(x_*, \xi_*, T-t_*) \in T^*\Omega^\circ$ before the first reflection.
    \begin{itemize}
        \item If $(x^T_P, \xi^T_P) \neq (x^T_S, \xi^T_S)$, then for any Gaussian beam solutions $u^P_h$ and $u^S_h$ with respect to spacetime geodesics from $(t_P, x_P, \tau_P, \xi_P)$ and $(t_S, x_S, \tau_S, \xi_S)$, respectively,
        \[
        \langle u^P_h, u^S_h \rangle = O(h^N),
        \]
        where $N$ is the construction error of the Gaussian beam.
        \item If $(x^T_P, \xi^T_P) = (x^T_S, \xi^T_S) = (x_0, \xi_0)$ and
        \[
        \sum_{\alpha=1}^n (A^\alpha_{SP} + \overline{A^\alpha_{PS}})(\xi_0)_\alpha \neq 0,
        \]
        then there exists Gaussian beam solutions $u^P_h$ and $u^S_h$ with respect to spacetime geodesics from $(t_P, x_P, \tau_P, \xi_P)$ and $(t_S, x_S, \tau_S, \xi_S)$, respectively, such that
        \[
        \langle u^P_h, u^S_h \rangle = h C + O(h^2), \quad C \neq 0.
        \]
    \end{itemize}
\end{theorem}

\begin{remark}
    An example where the assumption
    \[
        \sum_{\alpha=1}^n (A^\alpha_{SP} + \overline{A^\alpha_{PS}})(\xi_0)_\alpha \neq 0
    \]
    for all $x_0$ and almost every $\xi_0 \in T^*_{x_0}\Omega$ is simply $A_{PS} = A_{SP} = 1$. In fact, as long as
    \[
    (A^1_{SP} + \overline{A^1_{PS}},\, \cdots,\, A^n_{SP} + \overline{A^n_{PS}} )(x_0) \neq 0,
    \]
    then the non-vanishing requirement holds for almost every $\xi_0 \in T^*_{x_0}\Omega$.
\end{remark}

\begin{proof}
For $* = P, S$, denote $(t, x_*(t), \tau_*(t), \xi_*(t))$ the corresponding co-geodesic flow at time slice $t \geq t_*$. Then $(x_*(t_*), \xi_*(t_*)) = (x_*, \xi_*)$ and $(x_*(T), \xi_*(T)) = (x^T_*, \xi^T_*)$. Write
\[
u_h^P(T,x)
\sim
h^{-n/4}e^{i\phi_P(T,x)/h}
\left[
\begin{pmatrix}
b\\0
\end{pmatrix}
+
h
\begin{pmatrix}
*\\c
\end{pmatrix}
+\cdots
\right],
\]
\[
u_h^S(T,x)
\sim
h^{-n/4}e^{i\phi_S(T,x)/h}
\left[
\begin{pmatrix}
0\\d
\end{pmatrix}
+
h
\begin{pmatrix}
r\\ *
\end{pmatrix}
+\cdots
\right],
\]
where
\[
c=(a_1^P)_S,
\quad r=(a_1^S)_P.
\]
The $h^{-n/4}$ is for normalization. Then
\[
\begin{aligned}
\langle u_h^P(T),u_h^S(T)\rangle_{L^2}
&=
h^{-n/2}
\int e^{i(\phi_P-\overline{\phi_S})/h}
\left[
\begin{pmatrix}
b\\0
\end{pmatrix}
+
h
\begin{pmatrix}
*\\c
\end{pmatrix}
\right]
\cdot
\overline{
\left[
\begin{pmatrix}
0\\d
\end{pmatrix}
+
h
\begin{pmatrix}
r\\ *
\end{pmatrix}
\right]}
+O(h^2)\,dx.
\end{aligned}
\]
The leading polarizations are orthogonal:
\[
\begin{pmatrix}
b\\0
\end{pmatrix}
\cdot
\overline{
\begin{pmatrix}
0\\d
\end{pmatrix}}
=0.
\]
Therefore the actual leading term is
\[
h^{-n/2+1}
\int e^{i(\phi_P-\overline{\phi_S})/h}
\left(c\overline d+b\overline r\right)
\,dx.
\]
If $x_P^T \neq x_S^T$, then the cutoff in the Gaussian beam construction \eqref{eq: gaussian beam} will force the inner product of two Gaussian beam solutions to be 0 up to construction error $O(h^N)$. If $x_P^T = x_S^T = x_0$, then stationary phase may only occur at $x_0$ if
\[
\nabla (\phi_P - \overline{\phi_S})(x_0) = \xi_P - \xi_S
\]
is 0, namely $\xi_P = \xi_S$. Hence if $(x^T_P, \xi^T_P) \neq (x^T_S, \xi^T_S)$, then
\[
\langle u^P_h, u^S_h \rangle = O(h^N).
\]

Suppose now the $P$/$S$-beams intersect co-linearly at time $T$. Namely, they satisfy $(x^T_P, \xi^T_P) = (x^T_S, \xi^T_S) = (x_0, \xi_0)$, so that stationary phase occurs at $x_0$.
The relevant leakage amplitude is
\[
\kappa=c\overline d+b\overline r .
\]
Recall the formulas for $c$ and $r$:
\[
c
=
-i
\frac{A_{SP}(\xi_P)}
{(c_P^2-c_S^2)|\xi_P|^2}
\,b,
\qquad
r
=
i
\frac{A_{PS}(\xi_S)}
{(c_P^2-c_S^2)|\xi_S|^2}
\,d,
\]
where
\[
A_{SP}(\xi_P)
=
\sum_{\alpha=1}^n A^\alpha_{SP}\xi_{P,\alpha},
\qquad
A_{PS}(\xi_S)
=
\sum_{\alpha=1}^n A^\alpha_{PS}\xi_{S,\alpha}.
\]
Substitute in, we obtain
\[
\begin{split}
    \kappa(x_0) &=
    -\frac{i}{(c_P^2-c_S^2)|\xi_0|^2}
    A_{SP}(\xi_0)\,b\overline d
    +
    b\,
    \overline{
    \frac{i}{(c_P^2-c_S^2)|\xi_0|^2}
    A_{PS}(\xi_0)\,d
    }\\
    &= -\frac{i}{(c_P^2-c_S^2)|\xi_0|^2}(A_{SP}+ \overline{A_{PS}})(\xi_0)b\bar{d}.
\end{split}
\]
This term is nonzero if and only if
\[
(A_{SP} + \overline{A_{PS}})(\xi_0) \neq 0.
\]

For the specific stationary phase term, near $x_0$, write
\[
\Phi(x)
=
\Phi(x_0)
+
\frac12 (x-x_0)^T Q (x-x_0)
+
O(|x-x_0|^3),
\]
where $\Phi = \phi^P - \overline{\phi^S}$ and
\[
Q=
\partial_x^2\phi_P(T,x_0)
-
\overline{\partial_x^2\phi_S(T,x_0)}.
\]
By the construction of Gaussian beam solutions, both beams have positive imaginary quadratic parts,
\[
\operatorname{Im}\phi_P(T,x)\geq c|x-x_0|^2,
\qquad
\operatorname{Im}\phi_S(T,x)\geq c|x-x_0|^2,
\]
thus we have
\[
\operatorname{Im}\Phi(x)\geq c|x-x_0|^2.
\]
Therefore $Q$ is nondegenerate in the complex stationary phase sense. Consequently,
\[
h^{-n/2}
\int e^{i\Phi(x)/h}K(x)\,dx
\sim
e^{i\Phi(x_0)/h}
(2\pi)^{n/2}
\frac{\kappa(x_0)}
{\det(-iQ)^{1/2}}
+\cdots .
\]
Thus
\[
\langle u_h^P(T),u_h^S(T)\rangle_{L^2}
=
h\,
e^{i\Phi(x_0)/h}
(2\pi)^{n/2}
\frac{\kappa(x_0)}
{\det(-iQ)^{1/2}}
+O(h^2),
\]
has nonzero coefficient when
\[
    \sum_{\alpha=1}^n (A^\alpha_{SP} + \overline{A^\alpha_{PS}})(\xi_0)_\alpha \neq 0.
\]

\end{proof}

\begin{remark}\label{rem: spatial scan}
    Note that the full spacetime covectors are still different:
    \[
    \zeta_P=(\tau_P,\xi_0),
    \qquad
    \zeta_S=(\tau_S,\xi_0),
    \]
    with
    \[
    \tau_P=\pm c_P|\xi_0|,
    \qquad
    \tau_S=\pm c_S|\xi_0|.
    \]
    Thus $\zeta_P\neq \zeta_S$, but their spatial projections agree. As a result, if one uses spacetime inner product, there will be no stationary phase. This is the reason why having spatial $L^2$ inner product is so important here.
\end{remark}

We conclude here the specific geometric information recovered relating the two metrics. Specifically, suppose we detect that
\[
\langle u^P_h, u^S_h \rangle = Ch + O(h^2)
\]
for some $C \neq 0$, where $u^P_h, u^S_h$ are some Gaussian beam solutions from $(t_P, x_P, \tau_P, \xi_P)$ and $(t_S, x_S, \tau_S, \xi_S)$, respectively. Then by Theorem \ref{thm: leakage detection}, we know
for some unknown interior covector $(x_0, \xi_0)$,
\[
(x^T_P, \xi^T_P) = (x^T_S, \xi^T_S) = (x_0, \xi_0)
\]
for some unknown $(x_0, \xi_0)$. This means that
\begin{enumerate}
    \item \emph{the $P$-Hamiltonian flow from $(x_P, \xi_P)$ and the $S$-Hamiltonian flow from $(x_S, \xi_S)$ collide at $(x_0, \xi_0)$};
    \item \emph{the length of the corresponding $P$- and $S$-geodesics are $T-t_P$ and $T-t_S$, respectively;}
    \item \emph{the $*$-length of the covector $\xi_0$ is the same as the $*$-length of the initial covector $\xi_*$, for $* = P, S$, in particular the metric ratio can be read from initial data:
    \[
    \frac{c_P^2(x_0)}{c_S^2(x_0)} = \frac{c_P^2(x_P)|\xi_P|^2}{c_S^2(x_S)|\xi_S|^2}
    \]
    where $|\xi_0|$ stands for the usual Euclidean length.
    }
\end{enumerate}

\subsection{General $k$-velocity systems}\label{sec: general k-velocity systems}
Consider now the $k$-velocity system \eqref{eq: k wave equation}. Note that the polarization sets for the $j$-th wave is still in the simple form $e_j = (0, \cdots, 0, 1, 0, \cdots, 0)$. In particular, if we shoot out a $P_i$-beam and a $P_j$-beam, then their corresponding leading polarizations are still $b_ie_i$ and $b_je_j$ respectively, hence the top order interaction still occurs between the leading amplitude of one beam and subprincipal amplitude of the other. The proof is almost identical, with the relevant leakage amplitude being
\[
\kappa_{ij} = b_i\overline{(a_1^{P_i})_j} + (a_1^{P_j})_i\overline{b_j}
\]
where $a_1^{P_i}, a_1^{P_j}$ are the subprincipal level amplitudes. We refer to the proof of Theorem \ref{thm: leakage detection} for details.
\begin{theorem}\label{thm: k leakage detection}
    For $* = i,j$, let $x_* \in \partial \Omega$, $\xi_* \in T^*_{x_*}\Omega$ inward-pointing, $\tau_* = c_*|\xi_*|$, and $t_* \in (0, T)$ such that $(x^T_*, \xi^T_*):=\varphi_*(x_*, \xi_*, T-t_*) \in T^*\Omega^\circ$ before the first reflection.
    \begin{itemize}
        \item If $(x^T_i, \xi^T_i) \neq (x^T_j, \xi^T_j)$, then for any Gaussian beam solutions $u^i_h$ and $u^j_h$ with respect to spacetime geodesics from $(t_i, x_i, \tau_i, \xi_i)$ and $(t_j, x_j, \tau_j, \xi_j)$, respectively,
        \[
        \langle u^i_h, u^j_h \rangle = O(h^N),
        \]
        where $N$ is the construction error of the Gaussian beam.
        \item If $(x^T_i, \xi^T_i) = (x^T_j, \xi^T_j) = (x_0, \xi_0)$ and
        \[
        \sum_{\alpha=1}^n (A^\alpha_{ji} + \overline{A^\alpha_{ij}})(\xi_0)_\alpha \neq 0,
        \]
        then there exists Gaussian beam solutions $u^i_h$ and $u^j_h$ with respect to spacetime geodesics from $(t_i, x_i, \tau_i, \xi_i)$ and $(t_j, x_j, \tau_j, \xi_j)$, respectively, such that
        \[
        \langle u^i_h, u^j_h \rangle = h C + O(h^2), \quad C \neq 0.
        \]
    \end{itemize}
\end{theorem}

\section{Directional collision rigidity problem}\label{sec: directional collision rigidity problem}

We first define the geometric information related to a pair of metrics $(g_P, g_S)$.
\begin{definition}\label{def: directional collision relation}
    Let $(x_P, \xi_P), (x_S, \xi_S) \in \partial T^*\Omega$. For $* = P,S$, denote $H_*$ the Hamiltonian vector field of $g_*^{-1}(\xi, \xi)$, and $\varphi_*(x, \xi, \ell)$ the corresponding Hamiltonian flow parametrized by $\ell$ the length of projection with respect to $g_*$. We call $(x_P, \xi_P, \ell_P; x_S, \xi_S, \ell_S)$ a \emph{directional collision relation} if
    \[
    \varphi^P(x_P, \xi_P, \ell_P) = \varphi^S(x_S, \xi_S, \ell_S).
    \]
\end{definition}
Note that the relation is scaling invariant for covectors, namely
\[
(x_P, \xi_P, \ell_P; x_S, \xi_S, \ell_S)
\]
is a directional collision relation if and only if
\[
(x_P,\lambda \xi_P, \ell_P; x_S, \lambda \xi_S, \ell_S)
\]
is a directional collision relation for $\lambda > 0$. A natural question is the following rigidity problem.

\begin{problem}[directional collision rigidity problem]\label{problem: rigidity problem}
    Suppose $\Omega$ is a smooth manifold with boundary, and $(g_P, g_S)$ is a pair of metrics on $\Omega$. Moreover, suppose the metrics are given on $\partial T^* \Omega$. Does the directional collision relation $R$ uniquely determine the pair of metrics up to a common isometry?
\end{problem}

An immediate result is the following.
\begin{theorem}\label{thm: 1 metric known}
    Suppose the boundary is strictly convex with respect to both metrics and the geodesics are all non-trapping. If $g_P$ is a priori known and $g^{-1}_S|_{\partial T^* \Omega}$ is known, then $g_S$ is determined by the directional collision relation $R$.
\end{theorem}
\begin{proof}
    For any point and direction $(x_0, \xi_0)$, back track via $g_P$ to find the incoming direction $(x_P, \xi_P)$ and with $P$-length $\ell_P$. Denote $(x_S, \xi_S, \ell_S) = R(x_P, \xi_P, \ell_P)$, then
    \[
    g^{-1}_S(x_0)(\xi_0, \xi_0) = g_S^{-1}(x_S)(\xi_S, \xi_S).
    \]
\end{proof}

The interesting problem is whether the relation itself determines both metrics. Here we prove a simple version.
\begin{theorem}\label{thm: k collision rigidity}
    Suppose $k(k-1)/2 \geq n-1 \geq 1$, $\Omega \subset \mathbb{R}^n$, for $j = 1, \cdots, k$, $g_j = c_j^{-2}(x) dx^2$, $h_j = d_j^{-2}(x)dx^2$ conformal to Euclidean. Assume for any $i \neq j$, $c_i \neq c_j$, $d_i \neq d_j$, and for almost every $x \in \Omega$,
    \[
    \dim\mathrm{span}\{\nabla (d_i^2/d_j^2)(x): i \neq j\} \geq n-1.
    \]
    We assume the boundary is strictly convex to all metrics and all the geodesics are non-trapping. Then $R^g_{ij} = R^h_{ij}$ and $g^{-1}_j|_{\partial T^*\Omega} = h^{-1}_j|_{\partial T^*\Omega}$ for all $i, j$ implies $g_j = h_j$ on $\Omega$ for all $j$.
\end{theorem}
\begin{proof}
    Consider $\varphi^g_*$ the Hamiltonian flow with respect to $g_*$, and $\varphi^h_*$ the Hamiltonian flow with respect to $h_*$, parametrized by the length of the projection (we can also restrict to the cosphere bundle). Both are diffeomorphisms from subset of $\partial_{-,\bullet}T^* \Omega \times (0, \infty)$ to $T^* \Omega^\circ \backslash 0$ where $\partial_{-,\bullet} T^*\Omega$ is the set of inward pointing covectors on the boundary with respect to metric $\bullet = g_*,h_*$. Note that $R^g$ and $R^h$ maps in between these metric dependent inward pointing covectors. In fact, we have
    \[
    R^g_{ij} = (\varphi^g_j)^{-1} \circ \varphi^g_i, \quad R^h_{ij} = (\varphi^h_j)^{-1} \circ \varphi^h_i.
    \]
    As a result, the equality of directional collision relation gives
    \[
    \Phi_{ij} := \varphi^g_i \circ (\varphi^h_i)^{-1} = \varphi^g_j \circ (\varphi^h_j)^{-1}
    \]
    is a diffeomorphism on $T^*\Omega^\circ \backslash 0$ that simultaneously conjugates both flows. However, this holds for any $i, j$, hence $\Phi_{ij} = \Phi_{jk}$, and so on so forth. In other words, we obtain a diffeomorphism that simultaneously conjugates all the flows:
    \[
    \Phi := \varphi^g_j \circ (\varphi^h_j)^{-1}, \quad \forall j = 1, \cdots k. 
    \]
    So far, the special form of the metric is not used, and the goal is to show that this map is in fact fiber-wise constant.

    Denote now $\Phi(x, \xi) = (X(x, \xi), \Xi (x, \xi))$, and
    \[
    d\Phi = \begin{pmatrix}
        d_xX & d_\xi X\\
        d_x\Xi & d_\xi \Xi
    \end{pmatrix}.
    \]
    Since the metrics are conformally Euclidean, the Hamiltonian vector fields have a very simple form, namely for $g_j$ it is
    \[
    \hat Y_j = c_j^2 \xi \cdot \nabla_x - \frac{1}{2}|\xi|^2 \nabla c_j^2 \cdot \nabla_\xi,
    \]
    and same for $\hat Z_j$ corresponding to $h_j$. Since the flows are parametrized by projected arc length, the generator is
    \[
    Y_j = \frac{\hat Y_j}{c_j|\xi|} = c_j\frac{\xi}{|\xi|} \cdot \nabla_x - |\xi|\nabla c_j \cdot \nabla_\xi,
    \]
    and same for $Z_j$.
    By definition, we have
    \[
    d\Phi Z_j|_{(x, \xi)} = Y_j|_{\Phi(x,\xi)}. 
    \]
    Plug in, we obtain
    \[
    \begin{pmatrix}
        d_jd_xX\left(\frac{\xi}{|\xi|}\right) - |\xi| d_\xi X (\nabla d_j) \\ d_jd_x\Xi\left(\frac{\xi}{|\xi|}\right) - |\xi|d_\xi\Xi (\nabla d_j) 
    \end{pmatrix} = \begin{pmatrix}
        c_j(X(x, \xi))\frac{\Xi(x, \xi)}{|\Xi(x, \xi)|} \\ -|\Xi(x, \xi)|\nabla c_j(X(x, \xi))
    \end{pmatrix}.
    \]
    For the upper row, divide by $c_j(X(x, \xi))$ both sides, we obtain
    \[
    \frac{d_j(x)}{c_j(X(x, \xi))}d_xX\left(\frac{\xi}{|\xi|}\right) - |\xi| d_\xi X\left(\frac{\nabla d_j(x)}{c_j(X(x, \xi))}\right) = \frac{\Xi(x, \xi)}{|\Xi(x, \xi)|}.
    \]
    This is exactly the same for any other $i$. Subtracting the two, since they have the same RHS, we obtain
    \[
    \left(\frac{d_j(x)}{c_j(X(x, \xi))} - \frac{d_i(x)}{c_i(X(x, \xi))}\right) d_xX\left(\frac{\xi}{|\xi|}\right) = |\xi| d_\xi X\left(\frac{\nabla d_j(x)}{c_j(X(x, \xi))} - \frac{\nabla d_i(x)}{c_i(X(x, \xi))}\right).
    \]
    If we backtrack from $(x, \xi)$ with respect to $h_j$ and backtrack from $(X(x, \xi), \Xi(x, \xi))$ with respect to $g_j$, by definition of $\Phi$ we will arrive at the same boundary direction $(x_j, \xi_j)$. Same for $h_i$ and $g_i$, they arrive at the same boundary direction $(x_i, \xi_i)$. On the other hand, this backtracking gives us
    \[
    \frac{d_j^2(x)|\xi|^2}{d_i^2(x)|\xi|^2} = \frac{d_j^2(x_j)|\xi_j|^2}{d_i^2(x_i)|\xi_i|^2} = \frac{c_j^2(x_j)|\xi_j|^2}{c_i^2(x_i)|\xi_i|^2} = \frac{c_j^2(X(x, \xi))|\Xi(x, \xi)|^2}{c_i^2(X(x, \xi))|\Xi(x, \xi)|^2}.
    \]
    Hence the LHS of the previous equality is 0, and we have
    \[
    d_\xi X|_{(x, \xi)}\left[\frac{\nabla d_j(x)}{c_j(X(x, \xi))} - \frac{\nabla d_i(x)}{c_i(X(x, \xi))}\right] = 0.
    \]
    Use the metric ratio equality again, we obtain
    \[
    d_\xi X|_{(x, \xi)}\left[\frac{d_i}{c_i(X)}\left(\nabla \log \frac{d_i}{d_j} \right)\right] = 0
    \]
    and since the metric ratio is non-zero we obtain
    \[
    d_\xi X|_{(x, \xi)}\left[\nabla \frac{d_i^2}{d_j^2} \right] = 0.
    \]

    By assumption, the set of gradient of ratios have dimension at least $n-1$ for almost every $x \in \Omega$. We also have the $d_\xi X|_{(x, \xi)}[\xi] = 0$ simply because the collision relation is positive scaling invariant for covectors. Thus, for almost every $x \in \Omega$, and for almost every $\xi \in T^*_x\Omega$, we have
    \[
    \dim\mathrm{span}\{\xi, \nabla (d_i^2/d_j^2)(x)\} = n.
    \]
    In other words, for these $(x, \xi)$, we have $d_\xi X|_{(x, \xi)} = 0$. By smoothness this extends to all $(x, \xi)$, meaning $X$ is fiber-wise constant. By construction, $X$ can be extended to the boundary smoothly and is identity on the boundary. We thus obtain a boundary fixing diffeomorphism $X$.
    
    Next we show $X$ is an isometry for all $g_j$ and $h_j$ by showing $\Xi = (d_xX^{-1})^t$. Indeed, $\Phi$ is homogeneous so we can restrict it to a mapping between the two cosphere bundles. Since it is the composition of two Hamiltonian flows, it maps the natural contact form for $h_j$ to the natural contact form for $g_j$. Thus
    \[
    \xi_j dx^j = \Xi_j (d_xX)^j.
    \]
    In other words $\Xi = (d_xX^{-1})^t$, implying $\Phi$ is the natural lift of $X$. As $\Phi$ is by construction the conjugation of the two co-geodesic flows, this implies $X$ is an isometry.
    
    When $n \geq 3$, Liouville Theorem forces $X = \mathrm{Id}$ since it is identity on the boundary. While for $n = 2$, $X$ is a conformal diffeomorphism that preserves the orientation, hence it is holomorphic. By Maximum Modulus Principal, it has to be identity.
    
\end{proof}

The assumption $k(k-1)/2 \geq n-1$ is a necessary condition for the span requirement to hold. In the proof, one sees that the collision relation between 2 metrics can add 1-dimensional restriction on $d_\xi X$. Besides the radial direction, we still need at least $n-1$ directions to deduce $d_\xi X = 0$, so we require the number of pair of waves $k(k-1)/2$ to be at least $n-1$. Intuitively, if enough numbers of Hamiltonian flows are conjugated by the same cotangent bundle diffeomorphism, then we can upgraded it to a manifold diffeomorphism (under the span condition).

\section{An inverse problem for multi-velocity wave system}\label{sec: wave system inverse problem}

We combine the previous results to prove an inverse problem for $k$-velocity wave system.
\begin{theorem}\label{thm: k wave inverse problem}
    Consider \eqref{eq: k wave equation} for $k(k-1)/2 \geq n-1 \geq 1$ and for all $x \in \Omega$ and $i \neq j \in \{1,\cdots,k\}$,
    \[
    (A^1_{ij}+\overline{A^1_{ji}}, \,\cdots,\, A^n_{ij}+\overline{A^n_{ji}})(x) \neq 0.
    \]
    Suppose the jets of $c_j^2$ are known at the boundary $\partial \Omega$, the boundaries are strictly convex for all geodesics, all the geodesics are non-trapping, $T$ is larger than the length of any geodesic, and the first exit time function for each metric is known (see Remark \ref{rem: weaker assumption}). Suppose the span condition
    \[
    \dim \mathrm{span}\{\nabla (c_i^2/c_j^2)(x): i \neq j\} \geq n-1
    \]
    holds for almost every $x \in \Omega$. Then $C^T_I$ determines $c_j$ uniquely for all $j = 1, \cdots, k$ without any gauge.
\end{theorem}

\begin{remark}\label{rem: weaker assumption}
    Some of the assumptions can potentially be weakened.
    \begin{enumerate}
        \item One may not need all geodesics to be non-trapping, since the length of an open set of directions uniquely determines the metric at a point.
        \item The main reason for strict convexity is to avoid complicated glancing behaviors.
        \item The knowledge of first exit time for each geodesic is potentially not required, as it should be included in the canonical relation $\bigcup_{\alpha,\beta}\mathcal{C}_\beta^t \circ \mathcal{C}_\alpha$. However, mode conversion mixes the boundary data quite heavily, and it takes some effort to separate out the length function for each metric. We refer to \cite{HU03} for an example in Lam\'e system. Since the main goal of this paper is to identify the collision data from cross-sheet interaction, we directly assume the length function is given so that we can send beam solutions that reach time $T$ before any boundary reflection and mode conversion.
    \end{enumerate}
\end{remark}

\begin{proof}
    Since the jets of the metrics on the boundary are known, we can consider an extension of $\Omega$ and set up source term for Gaussian beams there. 
    
    
    With the knowledge of first exit time function, we can send beam solutions at sufficiently late times such that it stays in $\Omega^\circ$ before time $T$, which avoids reflection and mode conversion.
    By repeatedly construct Gaussian beam solutions and test their inner product, we collect all pairs $(t_i, x_i, \tau_i, \xi_i)$ and $(t_j, x_j, \tau_j, \xi_j)$ such that $\langle C^T_I f^i_h, f^j_h \rangle \neq O(h^\infty)$. By Theorem \ref{thm: k leakage detection}, these pairs correspond to collision data $(x_i, \xi_i, T-t_i; x_j, \xi_j, T-t_j)$. By assumption of $A$, for all $x \in \Omega^\circ$ and almost every $\xi \in T^*_x \Omega$, its corresponding $i,j$ collision data is recorded. In other words, for almost every $(x, \xi)\in T^*\Omega^\circ$, the corresponding $i,j$ collision data is recorded. Note that the collision relation is the graph of $\varphi_j^{-1} \circ \varphi_i$ where $\varphi_i$ is the Hamiltonian flow of $H_i$ parametrized by $g_i$-length of projection. Since the boundary is strictly convex, this map is smooth from subset of $\partial_{-,g_i}T^* \Omega \times (0, \infty)$ to $\partial_{-,g_j}T^*\Omega \times (0, \infty)$. Thus knowing this map for almost every point in the domain determines the entire map, the collision data for any $g_i$ and $g_j$ is determined.    
    
    Finally, as $k(k-1)/2 \geq n-1$ and the span condition is assumed, we can apply Theorem \ref{thm: k collision rigidity} to obtain the final result.
\end{proof}

\subsection{Cancellation for self-adjoint operators}\label{sec: cancellation wave}

One can think of the leakage detection condition 
\[
(A^1_{PS} + \overline{A^1_{SP}}, \, \cdots , \, A^n_{PS} + \overline{A^n_{SP}})(x) \neq 0
\]
as a first order condition. Namely, with this term being nonzero, we can observe the interaction between zeroth order amplitude and first order amplitude from $C^T_I$.
More generally, even if this term is 0, it is possible to have higher order leakage detection condition, so that even though the order $h$ term cancels out, the higher order terms do not cancel out perfectly.
From Theorem \ref{thm: microlocal}, the leakage detection condition is essentially the requirement that if $B=U^{-1}AU$ decouples microlocally, then $B^*U^*U \neq  U^*UB$ up to smoothing operators.

If we write the two-velocity wave model \eqref{eq: 2 wave equation} in the general $\partial_t^2 + A$ form, then
\[
A = -\begin{pmatrix}
    \nabla \cdot c_P^2 \nabla & 0\\ 0 & \nabla \cdot c_S^2 \nabla
\end{pmatrix} + \sum_{\alpha=1}^n A^\alpha \partial_\alpha + V.
\]
The adjoint with respect to standard $L^2(dx)$ is
\[
A^* = -\begin{pmatrix}
    \nabla \cdot c_P^2 \nabla & 0\\ 0 & \nabla \cdot c_S^2 \nabla
\end{pmatrix} - \sum_{\alpha=1}^n (A^\alpha)^* \partial_\alpha - \sum_{\alpha=1}^n \partial_\alpha (A^\alpha)^* +  V^*.
\]
Suppose $A$ is $M$-self-adjoint, then
\[
A^*M = MA \quad \mathrm{mod} \quad \Psi^{-\infty}.
\]
In particular, the principal symbol should match, so
\[
a_2^* M = Ma_2, \quad a_2 = \begin{pmatrix}
    c_P^2 & 0\\0 & c_S^2
\end{pmatrix}|\xi|^2 = a_2^*.
\]
As $c_P > c_S$ for all $x$, this forces $M$ to be a diagonal matrix
\[
M(x) = \begin{pmatrix}
    m_P(x)&0\\0&m_S(x)
\end{pmatrix}.
\]
The subprincipal symbol of $A$ and $A^*$ are
\[
a_1 = i\sum_{\alpha=1}^n A^\alpha \xi_\alpha, \quad a_1^* = -i\sum_{\alpha=1}^n (A^\alpha)^* \xi_\alpha.
\]
So for $A^*M = MA$ to hold on the subprincipal level, we need
\[
a_1^* M -\frac{1}{2i}\{a_2, M\} = Ma_1 - \frac{1}{2i}\{M, a_2\}.
\]
The off-diagonal terms simplify to
\[
m_PA^\alpha_{PS} + m_S \overline{A^\alpha_{SP}} = 0.
\]
Note that this is precisely what we need to prevent in order for leakage detection to be possible in Theorem \ref{thm: k wave inverse problem}, where we took the standard pairing with identity matrix so $m_P \equiv m_S \equiv 1$. In general, one can easily see from the proof that for general $M = \mathrm{diag}(m_P, m_S)$ positive definite, the (first order) leakage detection condition is
\[
(m_PA^1_{PS} + m_S \overline{A^1_{SP}}, \,\cdots,\, 
m_PA^n_{PS} + m_S \overline{A^n_{SP}})(x) \neq 0.
\]
Certainly, when $M$ is not diagonal, then $U^*MU$ does not have block form even on the principal order, so leakage detection is certainly possible. As diagonal $M$ is more subtle and has already included the standard $L^2(dx)$, We do not go into details here on non-diagonal $M$.

\begin{remark}\label{rmk: example self adjoint too strong}
    Note that the leakage detection condition only involves the off-diagonal entries $A^\alpha_{PS}$ and $A^\alpha_{SP}$  (off-block entries in general case), but $M$-self-adjoint requirement contains also requirements for diagonal entries $A_{PP}$ and $A_{SS}$ (block entries in general case). In other words, leakage detection fails when $A^*M-MA$ ($A$ is the operator) can be microlocally block-diagonalized with respect to characteristic sheets, even when $A$ is not completely $M$-self-adjoint. This is related to Remark \ref{rmk: self adjoint too strong}, justifying a weaker condition for vanishing cross-sheet interaction is
    \[
    \Pi_\alpha^*M\Pi_\beta = \Psi^{-\infty}, \quad \alpha \neq \beta.
    \]
\end{remark}

\section{Application to Lam\'e system}\label{sec: application to lame system}

In this section, we apply leakage detection to the dynamical isotropic elasticity system (Lam\'e system):
\begin{equation}\label{eq: lame system}
L_{\lambda,\mu,\rho}u
:=
\rho \partial_t^2 u-\nabla\cdot S^L(x,u)=0,
\end{equation}
where we take spatial dimension to be 3 and
\[
S^L_{ij}(x,u)
=
\lambda \widetilde\varepsilon_{mm}(u)\delta_{ij}
+
2\mu \widetilde\varepsilon_{ij}(u),
\qquad
\widetilde\varepsilon_{ij}(u)
=
\frac12(\partial_i u_j+\partial_j u_i).
\]

The isotropic Lam\'e system models the propagation of elastic waves in a solid medium. The unknown $u(t,x)\in \mathbb R^3$ denotes the displacement field of the material. The coefficient $\rho(x)>0$ is the mass density, while $\lambda(x)$ and $\mu(x)$ are the Lam\'e parameters describing the elastic response of the medium. In particular, $\mu$ is the shear modulus and measures the resistance of the material to shear deformation, while $\lambda$, together with $\mu$, determines the compressional response. The system supports two different types of elastic waves: pressure waves, or $P$-waves, traveling with speed $c_P^2$, and shear waves, or $S$-waves, traveling with speed $c_S^2$. This two-velocity structure is one of the main physical and mathematical features of elasticity. It reflects the fact that elastic media transmit both longitudinal and transverse disturbances, and it is precisely this multi-velocity structure that makes the inverse problem richer than the scalar wave case.

To see the two-velocity structure, after dividing by $\rho$, the principal part of the operator has the form
\[
\partial_t^2 u - \frac{\lambda + \mu}{\rho}\nabla(\nabla \cdot u) - \frac{\mu}{\rho}\Delta u.
\]
If we define the P- and S-wave speeds to be
\[
c_P^2=\frac{\lambda+2\mu}{\rho},
\qquad
c_S^2=\frac{\mu}{\rho},
\]
and the corresponding Riemannian metrics by $g_\bullet = c_\bullet^{-2} dx^2$ for $\bullet = P,S$, then the principal symbol has the form
\[
p(t, x, \tau, \xi) = -\tau^2 I + (c_P^2-c_S^2) \xi \otimes \xi + c_S^2 |\xi|^2 I.
\]
The eigenvalues are $-\tau^2 + c_P^2|\xi|^2$ with longitudinal subspace $\text{span}\{\xi\}$, and $-\tau^2 + c_S^2|\xi|^2$ on the transverse subspace $\xi^\perp$, with multiplicity 2.

\subsection{Gaussian beams and off-polarization leakage}\label{sec: lame beam solutions}

We compute the leakage detection condition with respect to standard $L^2(dx)$ pairing.
All the computations related to Gaussian beam solutions will be in Euclidean coordinates rather than the usual Fermi coordinates, for evolution equations of all the terms along the line, see the computations in \cite{UZ24}.
Even though Fermi coordinates are the standard coordinate choices for computing how terms evolve along the central ray, it is more convenient for us to directly compute everything in Euclidean coordinates to avoid complication caused by coordinate change when we are computing the stationary phase.
Thus throughout this computation
\[
\nabla=\nabla^E,\qquad \partial_i=\partial_i^E,\qquad 
\operatorname{div}=\operatorname{div}_E,
\]
and all dot products are Euclidean dot products. We use the standard $L^2(dx)$
pairing.

The Lam\'e operator is written in Euclidean coordinates as
\[
(Lu)_i
=
\rho\partial_t^2u_i
-
\partial_i(\lambda\partial_k u_k)
-
\partial_j\bigl(\mu(\partial_i u_j+\partial_j u_i)\bigr).
\]
Let
\[
u_h=e^{i\phi/h}(a_0+h a_1+\cdots),
\]
where $a_k$ are Euclidean vector fields. We write
\[
q_m:=a_m\cdot\nabla\phi .
\]
Then
\[
Lu_h
=
e^{i\phi/h}
\left(
h^{-2}I_1+i h^{-1}I_2+O(1)
\right),
\]
where
\[
I_1=P(\phi)a_0
\]
with
\[
P(\phi)a
=
-\rho\phi_t^2a
+
(\lambda+\mu)(a\cdot\nabla\phi)\nabla\phi
+
\mu|\nabla\phi|^2a.
\]
Thus explicitly
\[
I_1
=
-\rho\phi_t^2a_0
+
(\lambda+\mu)(a_0\cdot\nabla\phi)\nabla\phi
+
\mu|\nabla\phi|^2a_0.
\]
The first transport term is
\[
I_2=T[\phi,a_0]-iP(\phi)a_1,
\]
where
\[
\begin{aligned}
T[\phi,a_0]
={}&
\rho\phi_{tt}a_0
+
2\rho\phi_t\partial_ta_0
-
\nabla(\lambda a_0\cdot\nabla\phi)
-
\lambda(\operatorname{div}a_0)\nabla\phi
\\
&-
\operatorname{div}
\left[
\mu\left(
\nabla\phi\otimes a_0
+
a_0\otimes\nabla\phi
\right)
\right]
-
\mu\left[
\nabla a_0+(\nabla a_0)^T
\right]\nabla\phi .
\end{aligned}
\]
In components, with all derivatives Euclidean,
\[
\begin{aligned}
T^i[\phi,a_0]
={}&
\rho\phi_{tt}a_0^i
+
2\rho\phi_t\partial_ta_0^i
-
\partial_i(\lambda q_0)
-
\lambda(\partial_k a_0^k)\partial_i\phi
\\
&-
\partial_j
\left[
\mu\left(
(\partial_i\phi)a_0^j
+
(\partial_j\phi)a_0^i
\right)
\right]
-
\mu(\partial_j\phi)
\left(
\partial_i a_0^j+\partial_j a_0^i
\right).
\end{aligned}
\]

\subsubsection{S-beam}

We first compute the longitudinal subprincipal component of an S-beam. Consider a point on the central S-ray, choose a Euclidean orthonormal
frame
\[
e_1=\frac{\nabla\phi^S}{|\nabla\phi^S|},\qquad e_2,e_3\in e_1^\perp.
\]
Thus $e_1$ points in the ray direction. Write
\[
\partial_1=e_1\cdot\nabla,\qquad 
\partial_\alpha=e_\alpha\cdot\nabla,\qquad \alpha=2,3.
\]
Set
\[
\phi_t^S=-\sigma_S,\qquad 
\nabla\phi^S=\eta_S e_1,\qquad 
\sigma_S=c_S\eta_S.
\]
Thus
\[
\partial_1\phi^S=\eta_S,\qquad \partial_\alpha\phi^S=0.
\]
The S-polarization condition is
\[
a_0^S\cdot\nabla\phi^S=0.
\]
Hence along the central ray
\[
a_0^S=b_2e_2+b_3e_3.
\]
Introduce the abbreviations
\[
B_S:=\partial_2b_2+\partial_3b_3,
\quad
M_S:=\nabla_\perp\mu\cdot b
=
\sum_{\alpha=2}^3(\partial_\alpha\mu)b_\alpha,
\]
\[
C_S:=\nabla_\perp c_S\cdot b
=
\sum_{\alpha=2}^3(\partial_\alpha c_S)b_\alpha,
\quad
\Phi_S:=
\sum_{\alpha=2}^3b_\alpha\partial_1\partial_\alpha\phi^S.
\]
Differentiating the S-polarization condition gives
\[
\partial_k(a_0^S\cdot\nabla\phi^S)=0.
\]
On the ray this implies
\[
\eta_S\partial_k a_0^{S,1}
+
\sum_{\alpha=2}^3b_\alpha\partial_k\partial_\alpha\phi^S=0.
\]
In particular,
\[
\partial_1a_0^{S,1}=-\eta_S^{-1}\Phi_S.
\]
Also, from
\[
\phi_t^S=-c_S|\nabla\phi^S|=-c_S\eta_S,
\]
we get
\[
\partial_\alpha\phi_t^S
=
-\eta_S\partial_\alpha c_S
-
c_S\partial_1\partial_\alpha\phi^S.
\]
Differentiating the polarization condition in $t$ gives
\[
\eta_S\partial_ta_0^{S,1}
+
\sum_{\alpha=2}^3b_\alpha\partial_\alpha\phi_t^S=0.
\]
Therefore
\[
\partial_ta_0^{S,1}
=
C_S+c_S\eta_S^{-1}\Phi_S.
\]

We now evaluate the $e_1$-component of $I_2=0$. On the S-branch,
\[
\rho(\phi_t^S)^2=\mu|\nabla\phi^S|^2=\mu\eta_S^2.
\]
Thus $P(\phi^S)$ vanishes on the transverse space and acts on the longitudinal
component by
\[
P(\phi^S)(a_1^{S,1}e_1)
=
(\lambda+\mu)\eta_S^2a_1^{S,1}e_1.
\]
Hence the $e_1$-component of $I_2=0$ is
\[
T^1[\phi^S,a_0^S]-i(\lambda+\mu)\eta_S^2a_1^{S,1}=0.
\]
We compute $T^1[\phi^S,a_0^S]$ term by term. Since $a_0^{S,1}=0$,
\[
\rho\phi_{tt}^S a_0^{S,1}=0.
\]
The time-transport term gives
\[
2\rho\phi_t^S\partial_ta_0^{S,1}
=
2\rho(-c_S\eta_S)(C_S+c_S\eta_S^{-1}\Phi_S)
=
-2\rho c_S\eta_S C_S-2\mu\Phi_S.
\]
The term $-\partial_1(\lambda a_0^S\cdot\nabla\phi^S)$ vanishes on the ray.
The divergence term gives
\[
-\lambda(\operatorname{div}a_0^S)\partial_1\phi^S
=
-\lambda\eta_S(\partial_1a_0^{S,1}+B_S)
=
-\lambda\eta_SB_S+\lambda\Phi_S.
\]
Next,
\[
-\partial_j
\left[
\mu\left(
(\partial_1\phi^S)a_0^{S,j}
+
(\partial_j\phi^S)a_0^{S,1}
\right)
\right]
=
-\eta_SM_S-\mu\eta_SB_S+\mu\Phi_S.
\]
Finally,
\[
-\mu(\partial_j\phi^S)
\left(
\partial_1a_0^{S,j}+\partial_ja_0^{S,1}
\right)
=
-\mu\eta_S(2\partial_1a_0^{S,1})
=
2\mu\Phi_S.
\]
Adding these contributions,
\[
T^1[\phi^S,a_0^S]
=
(\lambda+\mu)\Phi_S
-
(\lambda+\mu)\eta_SB_S
-
\eta_SM_S
-
2\rho c_S\eta_SC_S.
\]
Therefore
\[
a_1^{S,1}
=
-i\eta_S^{-2}\Phi_S
+
i\eta_S^{-1}B_S
+
i\eta_S^{-1}\frac{M_S}{\lambda+\mu}
+
i\eta_S^{-1}\frac{2\rho c_S}{\lambda+\mu}C_S.
\]
Equivalently,
\[
a_1^{S,1}
=
-i\eta_S^{-2}
\sum_{\alpha=2}^3b_\alpha\partial_1\partial_\alpha\phi^S
+
i\eta_S^{-1}(\partial_2b_2+\partial_3b_3)
+
i\eta_S^{-1}\frac{\nabla_\perp\mu\cdot b}{\lambda+\mu}
+
i\eta_S^{-1}\frac{2\rho c_S}{\lambda+\mu}
\nabla_\perp c_S\cdot b .
\]

\subsubsection{P-beam}

We next compute the transverse subprincipal component of a P-beam. At the point
under consideration on the central P-ray, choose a Euclidean orthonormal frame
\[
e_1=\frac{\nabla\phi^P}{|\nabla\phi^P|},\qquad e_2,e_3\in e_1^\perp.
\]
Again $e_1$ points in the ray direction. Set
\[
\phi_t^P=-\sigma_P,\qquad 
\nabla\phi^P=\eta_Pe_1,\qquad 
\sigma_P=c_P\eta_P.
\]
Thus
\[
\partial_1\phi^P=\eta_P,\qquad \partial_\alpha\phi^P=0.
\]
We use the Euclidean longitudinal normalization
\[
a_0^P=\mathcal A\nabla\phi^P.
\]
Therefore along the ray
\[
a_0^P=\mathcal A\eta_Pe_1,
\qquad 
a_0^{P,1}=\mathcal A\eta_P,\qquad a_0^{P,\alpha}=0.
\]
Let
\[
H_{1\alpha}^P:=\partial_1\partial_\alpha\phi^P.
\]
From the P-eikonal equation
\[
\phi_t^P=-c_P|\nabla\phi^P|=-c_P\eta_P,
\]
we have
\[
\partial_\alpha\phi_t^P
=
-\eta_P\partial_\alpha c_P
-
c_PH_{1\alpha}^P.
\]
Since $a_0^{P,\alpha}=\mathcal A\partial_\alpha\phi^P$, this gives
\[
\partial_ta_0^{P,\alpha}
=
-A\eta_P\partial_\alpha c_P
-
Ac_PH_{1\alpha}^P.
\]

We evaluate the transverse component of $I_2=0$. On the P-branch,
\[
\rho(\phi_t^P)^2=(\lambda+2\mu)|\nabla\phi^P|^2=(\lambda+2\mu)\eta_P^2.
\]
The principal symbol vanishes on the longitudinal space and acts on transverse
components by
\[
P(\phi^P)(a_1^{P,\alpha}e_\alpha)
=
-(\lambda+\mu)\eta_P^2a_1^{P,\alpha}e_\alpha.
\]
Thus the transverse component of $I_2=0$ is
\[
T^\alpha[\phi^P,a_0^P]
+
i(\lambda+\mu)\eta_P^2a_1^{P,\alpha}=0.
\]
We now compute $T^\alpha[\phi^P,a_0^P]$. Since $a_0^{P,\alpha}=0$,
\[
\rho\phi_{tt}^Pa_0^{P,\alpha}=0.
\]
The time-transport term is
\[
2\rho\phi_t^P\partial_ta_0^{P,\alpha}
=
2\rho \mathcal A c_P\eta_P^2\partial_\alpha c_P
+
2\rho \mathcal A c_P^2\eta_PH_{1\alpha}^P.
\]
Since
\[
q_0^P=a_0^P\cdot\nabla\phi^P=\mathcal A\eta_P^2,
\]
we have
\[
-\partial_\alpha(\lambda q_0^P)
=
-\mathcal A\eta_P^2\partial_\alpha\lambda
-
\lambda\eta_P^2\partial_\alpha \mathcal A
-
2\lambda \mathcal A\eta_PH_{1\alpha}^P.
\]
The term
\[
-\lambda(\operatorname{div}a_0^P)\partial_\alpha\phi^P
\]
vanishes because $\partial_\alpha\phi^P=0$. Next,
\[
-\partial_j
\left[
\mu\left(
(\partial_\alpha\phi^P)a_0^{P,j}
+
(\partial_j\phi^P)a_0^{P,\alpha}
\right)
\right]
=
-2\mu \mathcal A\eta_PH_{1\alpha}^P.
\]
Finally,
\[
-\mu(\partial_j\phi^P)
\left(
\partial_\alpha a_0^{P,j}
+
\partial_j a_0^{P,\alpha}
\right)
=
-\mu\eta_P^2\partial_\alpha \mathcal A
-
2\mu \mathcal A\eta_PH_{1\alpha}^P.
\]
Adding everything, the $H_{1\alpha}^P$-terms cancel because
\[
\rho c_P^2=\lambda+2\mu.
\]
Hence
\[
T^\alpha[\phi^P,a_0^P]
=
-(\lambda+\mu)\eta_P^2\partial_\alpha \mathcal A
-
\mathcal A\eta_P^2\partial_\alpha\lambda
+
2\rho\mathcal  A c_P\eta_P^2\partial_\alpha c_P.
\]
Solving for the transverse component of $a_1^P$, we obtain
\[
a_1^{P,\alpha}
=
-i\partial_\alpha \mathcal A
-
\frac{i\mathcal A}{\lambda+\mu}\partial_\alpha\lambda
+
\frac{2i\rho c_P\mathcal A}{\lambda+\mu}\partial_\alpha c_P.
\]
Equivalently, using
\[
\rho c_P=\frac{\lambda+2\mu}{c_P},
\]
we may write
\[
a_1^{P,\alpha}
=
-i\partial_\alpha \mathcal A
-
\frac{i\mathcal A}{\lambda+\mu}
\left[
\partial_\alpha\lambda
-
2(\lambda+2\mu)\partial_\alpha\log c_P
\right],
\qquad \alpha=2,3.
\]

\subsubsection{Leading-amplitude contribution}

We now compute the leakage coefficient at a P/S collision point. Choose a common
Euclidean orthonormal coordinate system centered at the collision point $x_0$,
with $e_1$ equal to the common direction of the two rays and $e_2,e_3$ spanning
the common transverse plane. At $x_0$, impose stationary phase:
\[
\nabla\phi^P(x_0)=\nabla\phi^S(x_0)=\eta e_1.
\]
Thus $\eta_P=\eta_S=\eta$ at $x_0$. We write
\[
\Psi=\phi^P-\phi^S.
\]
The leading amplitudes are
\[
a_0^P=\mathcal A\nabla\phi^P,\qquad 
a_0^S=b_2e_2+b_3e_3,
\]
and the S-polarization condition is
\[
a_0^S\cdot\nabla\phi^S=0.
\]
Therefore near the ray,
\[
a_0^P\cdot a_0^S
=
\mathcal A\nabla\phi^P\cdot a_0^S
=
\mathcal A(\nabla\phi^P-\nabla\phi^S)\cdot a_0^S
=
\mathcal A a_0^S\cdot\nabla\Psi.
\]
Hence
\[
\int e^{i\Psi/h}a_0^P\cdot a_0^S\,dx
=
\int e^{i\Psi/h}\mathcal A a_0^S\cdot\nabla\Psi\,dx.
\]
Since
\[
\nabla e^{i\Psi/h}
=
\frac{i}{h}e^{i\Psi/h}\nabla\Psi,
\]
integration by parts gives
\[
\int e^{i\Psi/h}\mathcal A a_0^S\cdot\nabla\Psi\,dx
=
ih\int e^{i\Psi/h}\operatorname{div}(\mathcal Aa_0^S)\,dx.
\]
Thus the leading-amplitude contribution is
\[
L_{00}
=
i\operatorname{div}(\mathcal Aa_0^S)(x_0).
\]
Now
\[
\operatorname{div}(\mathcal Aa_0^S)
=
\nabla \mathcal A\cdot a_0^S
+
\mathcal A\operatorname{div}a_0^S.
\]
At $x_0$,
\[
\nabla \mathcal A\cdot a_0^S=\nabla_\perp \mathcal A\cdot b.
\]
Moreover,
\[
\operatorname{div}a_0^S
=
\partial_1a_0^{S,1}
+
\partial_2b_2+\partial_3b_3
=
-\eta^{-1}\Phi_S+B_S.
\]
Therefore
\[
L_{00}
=
i\nabla_\perp \mathcal A\cdot b
+
i\mathcal A B_S
-
i\mathcal A\eta^{-1}\Phi_S.
\]

\subsubsection{Subleading contribution}

With the standard $L^2(dx)$ pairing and phase $e^{i(\phi^P-\phi^S)/h}$, the
S-amplitude is conjugated. We assume here that $A$ and $b$ are real-valued
at the collision point. If $b$ is complex-valued, replace $b$ by $\overline b$
in the following formulas. The subleading interaction is
\[
L_{01}
=
a_1^P\cdot\overline{a_0^S}
+
a_0^P\cdot\overline{a_1^S}.
\]
Using the formula for $a_1^P$, we get
\[
a_1^P\cdot\overline{a_0^S}
=
-i\nabla_\perp \mathcal A\cdot b
-
\frac{i\mathcal A}{\lambda+\mu}\nabla_\perp\lambda\cdot b
+
\frac{2i\mathcal A(\lambda+2\mu)}{\lambda+\mu}
\nabla_\perp\log c_P\cdot b.
\]
Next, using the formula for $a_1^{S,1}$, we have
\[
\overline{a_1^{S,1}}
=
i\eta^{-2}\Phi_S
-
i\eta^{-1}B_S
-
i\eta^{-1}\frac{\nabla_\perp\mu\cdot b}{\lambda+\mu}
-
i\eta^{-1}\frac{2\rho c_S}{\lambda+\mu}
\nabla_\perp c_S\cdot b.
\]
Since
\[
a_0^P=\mathcal A\eta e_1,
\]
we obtain
\[
\begin{aligned}
a_0^P\cdot\overline{a_1^S}
={}&
\mathcal A\eta\,\overline{a_1^{S,1}}
\\
={}&
i\mathcal A\eta^{-1}\Phi_S
-
i\mathcal AB_S
-
\frac{i\mathcal A}{\lambda+\mu}\nabla_\perp\mu\cdot b
-
\frac{2i\mathcal A\rho c_S}{\lambda+\mu}
\nabla_\perp c_S\cdot b.
\end{aligned}
\]
Therefore
\[
\begin{aligned}
L_{01}
={}&
-i\nabla_\perp \mathcal A\cdot b
-
i\mathcal AB_S
+
i\mathcal A\eta^{-1}\Phi_S
\\
&-
\frac{i\mathcal A}{\lambda+\mu}\nabla_\perp(\lambda+\mu)\cdot b
+
\frac{2i\mathcal A(\lambda+2\mu)}{\lambda+\mu}
\nabla_\perp\log c_P\cdot b
\\
&-
\frac{2i\mathcal A\rho c_S}{\lambda+\mu}
\nabla_\perp c_S\cdot b.
\end{aligned}
\]

\subsection{Standard $L^2$ leakage detection}

Adding up $L_{00}$ and $L_{01}$, the terms containing $\nabla_\perp \mathcal A, B_S$ and $\Phi_S$ cancel. Hence
\[
\begin{aligned}
L_{PS}
:={}&L_{00}+L_{01}
\\
={}&
-\frac{i\mathcal A}{\lambda+\mu}\nabla_\perp(\lambda+\mu)\cdot b
+
\frac{2i\mathcal A(\lambda+2\mu)}{\lambda+\mu}
\nabla_\perp\log c_P\cdot b
-
\frac{2i\mathcal A\rho c_S}{\lambda+\mu}
\nabla_\perp c_S\cdot b\\
={}& \frac{i\mathcal A}{\lambda+\mu}
\left[
-\nabla_\perp(\lambda+\mu)
+
(\lambda+2\mu)\nabla_\perp\log c_P^2
-
\mu\nabla_\perp\log c_S^2
\right]\cdot b.
\end{aligned}
\]
The terms inside the bracket collapse into
\[
\begin{aligned}
&-\nabla(\lambda+\mu)
+
(\lambda+2\mu)\nabla\log c_P^2
-
\mu\nabla\log c_S^2
\\
&=
-\nabla(\lambda+\mu)
+
(\lambda+2\mu)
\left(
\frac{\nabla(\lambda+2\mu)}{\lambda+2\mu}
-
\nabla\log\rho
\right)
-
\mu
\left(
\frac{\nabla\mu}{\mu}
-
\nabla\log\rho
\right)
\\
&=
-\nabla(\lambda+\mu)
+
\nabla(\lambda+2\mu)
-
\nabla\mu
-
(\lambda+\mu)\nabla\log\rho
\\
&=
-(\lambda+\mu)\nabla\log\rho.
\end{aligned}
\]
Thus, in the standard $L^2(dx)$ pairing,
\[
L_{PS}
=
-i\mathcal A\,\nabla_\perp\log\rho\cdot b.
\]
Equivalently, writing $\Pi_{\xi^\perp}$ for the Euclidean projection onto the
orthogonal complement of the common covector direction at the collision point,
\[
L_{PS}
=
-i\mathcal A\,\Pi_{\xi^\perp}\bigl(\nabla^E\log\rho\bigr)\cdot b.
\]

\subsection{Inverse problem for Lam\'e system}
With the above computation, we can obtain similar results as the wave system. However, Lam\'e system has $k=2$ waves with $n=3$, so $k(k-1)/2 = 1 < 2 = n-1$, hence currently we need to have the knowledge of one of the metrics in order to recover the other using Theorem \ref{thm: 1 metric known}.
\begin{theorem}\label{thm: lame}
    Suppose one of $c_P^2$ or $c_S^2$ is a priori known, the jets of both metrics are known on the boundary, and $\nabla^E \rho \neq 0$ for all $x \in \Omega$. Suppose the boundaries are strictly convex with respect to both metrics, all the geodesics for both metrics are non-trapping, $T$ is larger than the length of any geodesic, and the first exit time function for $P$ and $S$ are known (see Remark \ref{rem: weaker assumption}). Then $C^T_I$ uniquely determines the other metric.
\end{theorem}
\begin{proof}
    It suffices to show that for almost every $(x_0, \xi_0) \in T^*\Omega^\circ$, the leakage detection condition holds and can be detected through Gaussian beams with properly set initial data. The rest of the proof is identical to Theorem \ref{thm: k wave inverse problem} except that Theorem \ref{thm: 1 metric known} is used rather than Theorem \ref{thm: k collision rigidity}.

    By assumption, for all $x_0 \in \Omega^\circ$ and almost every $\xi_0 \in T_{x_0}^*\Omega$, we have
    \[
    \Pi_{\xi^\perp} (\nabla^E \log \rho) \neq 0.
    \]
    Recall that the leakage detection condition is
    \[
    L_{PS} = -i\mathcal A\Pi_{\xi^\perp} (\nabla^E \log \rho) \cdot b \neq 0
    \]
    where $\mathcal A$ is the leading amplitude coefficient for $P$-beam and $b$ is the leading amplitude for $S$-beam which is a two-vector living in $\xi^\perp$. Since $b$ and $\mathcal A$ can be chosen arbitrarily by varying the initial data (see \cite{UZ24}), there exists some Gaussian beam solutions such that $L_{PS} \neq 0$. Hence leakage detection condition holds for these $(x_0, \xi_0)$ and the interaction can be observed through $C^T_I$.
\end{proof}

\subsection{Cancellation under density-weighted measure}\label{sec: cancellation lame}
Since the Lam\'e system is self-adjoint, Blagovestchenskii identity provides the density-weighted $L^2$ inner product, instead of the standard $L^2$.
\begin{lemma}[Lemma 4 \cite{BL02}]
    For Lam\'e system, the $\rho$-weighted inner product of two solutions at time $T > 0$, namely
    \[
    \langle C^T_\rho f_1, f_2\rangle :
    = \int_\Omega u^{f_1}(T, x) u^{f_2}(T, x) \rho(x) dx
    \]
    for $f_1$ and $f_2$ boundary Dirichlet data supported in $(0, T) \times \partial \Omega$, can be determined from the Dirichlet-to-Neumann map $\Lambda$.
\end{lemma}

In this case, we have
\[
L_{01}^\rho = \rho L_{01},
\]
but during the computation of leading amplitude contribution we will have
\[
L_{00}^\rho = i \,\mathrm{div}(\rho Aa_0^S)(x_0) = \rho L_{00} + iA\nabla_\perp \rho \cdot b.
\]
As a result,
\[
L_{PS}^\rho = L_{00}^\rho + L_{01}^\rho = \rho L_{PS} + iA \nabla_\perp \rho \cdot b = 0,
\]
and the first order interaction is not observed from $C^T_\rho$. One can compute lower order interactions, and all of these will cancel out essentially because of the self-adjoint cancellation in the microlocal decoupling procedure. From another point of view: the canonical relation for the DN map only contains the scattering relation related to reflection and mode conversion on the boundary, and Blagoveshchenskii identity shows that $C^T_\rho$ is determined by the DN map, hence it is reasonable that $C^T_\rho$ does not contain cross-sheet interaction data which occurs in the interior.

\bibliographystyle{abbrv}
\bibliography{ref}

\end{document}